\documentclass[12pt]{article}
\usepackage{amsmath,amssymb,amsthm,bbm,mathabx,xcolor,graphicx,booktabs,ulem,natbib,xargs,fullpage}
\usepackage[english]{babel}     
\usepackage[utf8]{inputenc}
\usepackage[colorlinks]{hyperref}
\usepackage[shortlabels]{enumitem}
\usepackage{authblk}

\newcommand\iid{i.i.d.}
\newcommand\ie{i.e.}
\newcommand\wrt{w.r.t.}

\newcommand\bigo{O}
\newcommand\smallo{o}

\newcommand\gradientnoise{e}
\newcommand{\GNZ}{\boldsymbol{\zeta}}
\newcommand\id{\mathrm{I}}

\newcommand\Rset{\mathbb{R}}

\newcommand\bstheta{\boldsymbol{\theta}}
\newcommand\bszero{\mathbf{0}}

\newcommand\eqdistr{\stackrel{\mbox{\tiny\rm d}}{=}}  
\newcommand\convvague{\stackrel{\tiny{v}} {\longrightarrow}} 

\newcommand\rme{\mathrm{e}}

\newcommand\rmd{\mathrm{d}} 

\newcommand\esp{\mathbb E}
\newcommand\pr{\mathbb P}

\newcommand\ind[1]{\mathbbm{1}{\left\{#1\right\}}}

\newcommandx\quantilefunction[2][2=]{{#1}_{#2}^\leftarrow}

\newcommand\reduline{\bgroup\markoverwith
  {\textcolor{red}{\rule[-0.4ex]{2pt}{1.5pt}}}\ULon}

\newcommand\redsout{\bgroup\markoverwith {\textcolor{red}{\rule[0.4ex]{2pt}{1.5pt}}}\ULon}

 \usepackage{cleveref}
\newtheorem{theorem}{Theorem}[section]
\newtheorem{lemma}[theorem]{Lemma}

\newtheorem{assumption}[theorem]{Assumption}
\newtheorem{definition}[theorem]{Definition}

\newtheorem{example}[theorem]{Example}
\theoremstyle{remark}
\newtheorem{remark}[theorem]{Remark}

\crefname{assumption}{Assumption}{Assumptions}
\crefname{theorem}{Theorem}{Theorems}
\crefname{proposition}{Proposition}{Propositions}
\crefname{corollary}{Corollary}{Corollaries}

\crefname{equation}{}{}
\crefname{enumi}{}{}

\numberwithin{equation}{section}

\begin{document}

\title{Convergence of Stochastic Gradient Descent with
mini-batching and infinite variance}

\author[1,2]{Bartosz Głowacki}
\author[1]{Rafał Kulik}
\author[2]{Philippe Soulier}

\affil[1]{Department of Mathematics and Statistics, University of Ottawa}
\affil[2]{MODAL'X, Université Paris Nanterre}

\affil[ ]{
\texttt{bglowack@uottawa.ca},
\texttt{rkulik@uottawa.ca},
\texttt{psoulier@parisnanterre.fr}
}
\date{}

\maketitle


\begin{abstract}
  Stochastic gradient descent (SGD) with mini-batching is a standard tool in large-scale
  optimization, yet its theoretical properties under heavy-tailed gradient noise remain
  largely unexplored. In this paper we study SGD with increasing batch sizes when the
  gradient noise belongs to the domain of attraction of an $\alpha$-stable law with
  $\alpha\in(1,2)$. Building on existing results for the finite-variance regime and for
  heavy-tailed SGD without batching, we establish three main results. First, we derive
  $L^p$ moment bounds for the SGD error and show that increasing batch sizes lead to
  faster convergence rates. In particular, batching enables convergence in probability
  even for a constant stepsize. Second, we prove that the properly normalized SGD iterates
  converge in distribution to the stationary law of an Ornstein–Uhlenbeck process driven
  by an $\alpha$-stable Lévy process. Third, for Polyak–Ruppert averaging we obtain a
  stable limit theorem with a normalization that explicitly depends on the batch-size
  schedule.
\end{abstract}

\setcounter{tocdepth}{1}

\section{Introduction}
\label{sec:intro}

Stochastic gradient descent (SGD) is a class of stochastic optimization methods with a
number of theoretical results that establish moment bounds, convergence of the algorithm
in distribution or convergence of averages (\cite{robbins:munro:1951},
\cite{polyak-youditski:1992}, \cite{pelletier:1998}, \cite{kushner:yin:1997}). It has
become one of the most widely used algorithms for large-scale optimization in statistics
and machine learning. It is typically applied to problems where one aims to minimize an
objective function defined as an expected loss, or its empirical counterpart built from
data. In many modern applications the dataset is large, so computing the full gradient of
the objective at every iteration is computationally prohibitive. SGD addresses this by
replacing the full gradient with a stochastic estimate computed from a small subset of
data, leading to inexpensive iterations.

A common practical variant is a mini-batch SGD, where the stochastic gradient is formed by
averaging gradients over a batch of samples. In the classical finite-variance setting,
mini-batching is known to reduce the variability of the updates and improve stability, and
it is often analyzed through its effect on convergence rates and asymptotic variance
(\cite{gadat:gavra:2022}, \cite{leroux:2012}, \cite{johnson:zhang:2013},
\cite{juditsky:2024}, \cite{ma:2022}, \cite{umeda:Iiduka:2025}, \cite{kamo:Iiduka:2025}).

Theoretical results for SGD typically require light tails, in particular finite variance
assumptions. At the same time, it has been empirically observed that the gradient noise
may exhibit heavy tails. This led to some attempts to explain theoretically a source of
the heavy tails, as well as to preliminary theoretical results
(\cite{wang:gurbuzbalanan:shu:simsekli:erdogdu:2021},
\cite{gurbuzbalanan:simsekli:zhu:2021}, \cite{simsekli:2019}). It should be pointed out
that these studies are done in the context of online (one-pass) setting, since in the offline
case it is rather unclear what heavy tails could refer to (see
\cite{pavasovic:durmus:simsekli:2023}).

When studying theoretical performance of SGD, one is typically interested in three types
of results: a) moment bounds for the difference $\bstheta_N-\bstheta^*$, where
$\bstheta_N$ is the value of the $N$th iterate, while $\bstheta^*$ is the (unique)
minimizer; b) weak convergence of $w_N\left(\bstheta_N-\bstheta^*\right)$ (with
$w_N\to\infty$); c) weak convergence of $c_N^{-1}\sum_{i=1}^N(\bstheta_i-\bstheta^*)$.

To describe briefly the existing results, let us start with light tails. For a), the
results go back to the seminal works of \cite{robbins:munro:1951} and \cite{chung:1954}, with a plethora
of extensions (see e.g. \cite{moulines:bach:2011,
  nemirovski:juditsky:lan:shapiro:2009}), including mini-batching
(\cite{umeda:Iiduka:2025}, \cite{kamo:Iiduka:2025}, \cite{li:zhang:chen:smola:2014}). The
weak convergence question is typically established using the embedding technique
(\cite{pelletier:1998}), with the resulting rate of convergence being
$w_N=\sqrt{\gamma_N^{-1}}$, where $\{\gamma_i\}_{i\geq 1}$ is the learning
sequence. Finally, the Polyak-averaging in c) reduces the asymptotic variance and yields
the optimal convergence rates $c_N=\sqrt{N}$ and asymptotic normality
(\cite{polyak-youditski:1992}).

In the case of heavy tails the literature is very recent and considers online (one-pass)
scenario only. In the online SGD, heavy tails may stem from two sources. In the first
scenario one assumes that the learning rate is constant, then $\bstheta_N$ converges to a
random element $\bstheta_{\infty}$ that can be heavy-tailed, even if the input data are
light-tailed. This phenomenon has been well-known for a long time in the time series
literature (see \cite{buraczewski:damek:mikosch:2016}). In the context of SGD it was
re-discovered in \cite{gurbuzbalanan:simsekli:zhu:2021} with some analysis on the relation
between the learning rate, the batch size and the resulting tail index.

In the second scenario, heavy tails in the gradient noise stem from the input data.  In
this context, moment bounds for SGD were established in
\cite{wang:gurbuzbalanan:shu:simsekli:erdogdu:2021}, who proved $L^p$ bounds for the error
$\bstheta_N-\bstheta^*$. In this case the rate is of order $\bigo(\gamma_N^{p-1})$
for $p \in (1,2)$. Distributional limits for the normalized iterates were obtained by
\cite{blanchet:2024}, who showed that the properly rescaled SGD iterates converge to the
stationary distribution of an Ornstein--Uhlenbeck process driven by an $\alpha$-stable
L\'evy process. The convergence of averages is considered in
\cite{wang:gurbuzbalanan:shu:simsekli:erdogdu:2021}. Again, the limiting law is stable.

In this paper we consider online SGD with heavy tails and mini-batching. In order to be
able to study asymptotic behaviour, we make mini-batches $\{b_i\}_{i\geq 1}$
time-dependent. We cover three theoretical aspects. We start with the moment bounds
(\Cref{sec:moment-bound}), extending results from
\cite{wang:gurbuzbalanan:shu:simsekli:erdogdu:2021} to the mini-batching scheme. The rate
of convergence changes to $\bigo((\gamma_N/b_N)^{p-1})$, under similar assumptions.
Next in \Cref{sec:convergence}, we study weak convergence of the algorithm, extending the
results in \cite{blanchet:2024}. Again, the rate of convergence changes from
$\gamma_N^{1/\alpha-1}$ to $(\gamma_N/b_N)^{1/\alpha-1}$ (with an additional slowly
varying factor). The assumptions in our paper are similar to those in
\cite{blanchet:2024}. Finally, in \Cref{sec:averaging} we consider the averaging scheme,
with the rate of convergence $c_N$ being $(\sum_{i=1}^N b_i^{1-\alpha})^{1/\alpha}$, again
up to a slowly varying function.  The assumptions for the averaging scheme are much weaker
as compared to \cite{wang:gurbuzbalanan:shu:simsekli:erdogdu:2021}.

On the technical side, for the convergence of the algorithm, we proceed in a different way
as compared to \cite{blanchet:2024}, where the proof utilizes an embedding technique and
convergence to the stationary distribution of a L\'{e}vy-driven, Ornstein-Uhlenbeck process
(following the methodology from \cite{pelletier:1998}). Instead, we approximate
$w_N\left(\bstheta_N-\bstheta^*\right)$ by a weighted sum of independent regularly varying
random vectors. This way, the techniques of the proof for the weak convergence of the
algorithm and the weak convergence of the averages are virtually the same, making
extensive use of the modern theory of heavy tailed time series (see
\cite{kulik:soulier:2020}). Also, we believe, it opens the door to studying much more
complicated versions of heavy-tailed SGD.

We note in passing that we are not aware of any results that study convergence of the
algorithm $w_N\left(\bstheta_N-\bstheta^*\right)$ neither convergence of the averages in
case of batching and finite variance case. One can argue that in the former case, the
convergence can be easily obtained by introducing an additional sequence
$\{\sigma_i\}_{i\geq 1}$ that scales non-batched gradient noise:
$\sigma_i=b_i^{-1/2}$. Then, we can utilize the results in \cite{pelletier:1998}. However,
in the context of regular variation this additional sequence would involve the unknown
$\alpha$ and a slowly varying function. Therefore, in case of heavy tails one needs to
carefully consider the structure of the mini-batches and perform ``de-batching'', i.e.,
representing the mini-batch gradient as a weighted sum of individual gradients.

The paper is structured as follows. In \Cref{sec:prel} we introduce the model and set up
the notation. \Cref{sec:moment-bound} deals with the moment bound.  Convergence of the
algorithm is proved in \Cref{th:stable_law_convergence}, while
\Cref{th:averaging_stable_convergence} gives the limiting stable law for the averaging
scheme. 
Technical lemmas are postponed to the appendix.


\section{Preliminaries}\label{sec:prel}
In this section we introduce the notation and formulate the SGD problem. The precise assumptions
will be stated in each section separately, when dealing with moment bounds
(\Cref{sec:moment-bound}), convergence of the algorithm (\Cref{sec:convergence}) or convergence of
averages (\Cref{sec:averaging}).

Let $\psi:\Rset^d\times \Rset^d\to\Rset$ be a differentiable function.  Let $\xi_j$, $j=1,\ldots,n$
be a sample from the distribution of $\xi_0$, where $\xi_0$ is a random vector in
$\Rset^d$. 
Assume that $\esp[\nabla\psi(\bstheta;\xi_0)]<\infty$ and define $H:\Rset^d\to\Rset^d$ by
\begin{align*}
H(\bstheta)=\esp[\nabla\psi(\bstheta;\xi_0)]\;,
\end{align*}
where $\esp$ is the expectation with respect to the distribution of $\xi_0$. Let $\bstheta^*$ be a
root (not necessarily unique) of $H$:
\begin{align}
  \label{eq:zero-problem-population}
  H(\bstheta^*)=\bszero\;,
\end{align}
where $\bszero\in\Rset^d$ is the zero vector.  Our assumptions will guarantee that the population
minimizer is unique. This differs from the standard formulation in the SGD literature. Indeed, it is typically assumed
that
\begin{align}
  \label{eq:minimization-problem-population}
  \bstheta^*={\rm argmin}_{\bstheta \in \Rset^d} \Phi(\bstheta), 
\end{align}
where 
\begin{align*}
  \Phi(\bstheta):=\esp[\psi(\bstheta;\xi_0)]\;.
\end{align*} 
However, in the context of heavy tails, it may happen that $\esp[\psi(\bstheta;\xi_0)]$ is infinite,
while $\esp[\nabla\psi(\bstheta;\xi_0)]<\infty$. See \Cref{xmpl:linear-regression} below. 
On the other hand, when $\Phi$ is well defined, then under some usual conditions we have $\esp[\nabla\psi(\bstheta;\xi_0)]=\nabla \esp[\psi(\bstheta;\xi_0)]$, both gradients are equal, $\nabla\Phi=H$ and the minimizer of $\Phi$ is the zero of $H$. The usual conditions above are e.g. 
\begin{itemize}
\item For almost every $\xi_0$, the function $\bstheta \mapsto \psi(\bstheta;\xi_0)$ 
is differentiable;

\item There exists an integrable random variable $C(\xi_0)$ such that
for all $\bstheta$, $\|\nabla \psi(\bstheta;\xi_0)\| \le C(\xi_0)$ and 
$\mathbb{E}[C(\xi_0)] < \infty$.
\end{itemize}
We want to produce a random sequence $\{\bstheta_i, i=1,\ldots,N\}$, that converges (in some sense,
as $N\to\infty$) to $\bstheta^*$ defined as a solution to \eqref{eq:zero-problem-population} (or
defined in \eqref{eq:minimization-problem-population}). For this, we will perform the stochastic
gradient descent algorithm, where at the iteration $i+1$ we will approximate $H(\bstheta_i)$ using
the empirical averages with growing batch sizes.  Let $\{b_i\}_{i\geq 1}$ be a non-decreasing
sequence of positive integers. Denote then
\begin{align*}
  \Omega_{i+1}=\left\{\sum_{k=1}^ib_k+1,\ldots,\sum_{k=1}^{i+1}b_k\right\},\ \ i=0,\ldots,N-1\;.
\end{align*}
Hence, $|\Omega_{i}|=b_i$. At time $i+1$ we approximate $H(\bstheta_i)$ by
$\widetilde{H}_{i+1}(\bstheta_i)$:
\begin{align*}
  \widetilde{H}_{i+1}(\bstheta)=\frac{1}{b_{i+1}}\sum_{j\in \Omega_{i+1}}\nabla\psi(\bstheta;\xi_j)\;. 
\end{align*}
Furthermore, let $\{\gamma_i\}_{i\geq 1}$ be a non-increasing positive sequence called the learning
rate sequence.  The stochastic gradient descent (SGD) algorithm with batching is defined through
the following recursive procedure:
\begin{itemize}
\item Start with $\bstheta_0$.
\item For $i=0,\ldots,N-1$,
  \begin{align}
    \label{eq:SGD}
    \bstheta_{i+1}=\bstheta_i -\gamma_{i+1} \widetilde{H}_{i+1}(\bstheta_i)\;.
  \end{align}
\end{itemize} 
The online sequence $\{\bstheta_{i}\}_{i\geq 1}$ depends on the past data
$\xi_i,0\leq i\leq b_1+\cdots+b_i$ that is, $\bstheta_i$ is $\mathcal{F}_{i}$-measurable, where
$\mathcal{F}_{i}=\sigma(\xi_j,j=0,\ldots,\sum_{k=1}^ib_k)$. This algorithm is said to be
single-pass, that is, for each iteration we use only one block of data.

The convergence properties of the algorithm will depend on the assumptions on the gradient noise
sequence defined as follows.
\begin{definition}
  \label{def:gradient-noise-sequence}
  \index{gradient~noise~sequence} The gradient noise sequence
  $\{\gradientnoise_i(\bstheta)\}_{i\geq 1}$ is defined as
\begin{align*}
  \gradientnoise_{i+1}(\bstheta)
  =\frac{1}{b_{i+1}}\sum_{j\in\Omega_{i+1}}\left\{\nabla \psi(\bstheta;\xi_j)
  -\esp[\nabla \psi(\bstheta;\xi_j)]\right\}=\widetilde{H}_{i+1}(\bstheta)-H(\bstheta)\;.
\end{align*}
\end{definition}
Thus, the SGD algorithm is written as follows:
\begin{align}
  \label{eq:sgd_algorithm}
  \bstheta_{i+1}=\bstheta_i-\gamma_{i+1}H(\bstheta_i)-\gamma_{i+1}\gradientnoise_{i+1}(\bstheta_i)\;, \ \
  i=0,\ldots,N-1\;.
\end{align}
We will state our assumptions in the language of the gradient noise, instead of $\nabla \psi(\bstheta;\xi_j)$. It will allow for potential generalizations, such as dependence between $\psi(\bstheta;\xi_j)$, $j\in \Omega_{i+1}$.\\

We will use the following notation:
\begin{itemize}
\item $N$ for the number of iterations,
\item $n=b_1+\cdots+b_N$ for the total sample size,
\item the index $i=1,\ldots,N$ for the iterations,
\item $i=1,\ldots,n$ for data index of $\xi_j$.
\end{itemize}
\begin{example}\label{xmpl:linear-regression}{\rm 
  Consider the univariate linear regression model
  \begin{align*}
    Y_i=\bstheta^* X_i+\epsilon_i\;, \ \ i\geq 0\;,
  \end{align*} 
  where $\{(X,\epsilon),(X_i,\epsilon_i),i\geq 0\}$ are \iid~random vectors with $\esp[X^2]<\infty$,
  $\esp[|\epsilon|]<\infty$, $\esp[\epsilon]=0$ and $\esp[\epsilon^2]=\infty$. In this case
  $\xi_i:=(X_i,Y_i)$.  Consider the square loss function
  $\psi(\bstheta;\xi)=\frac{1}{2}\left(Y-\bstheta X\right)^2$.  Then $\Phi$ does not exist. On the
  other hand
  \begin{align*}
    &\nabla\psi(\bstheta,\xi)=-X(Y-\bstheta X)=X^2(\bstheta-\bstheta^*)+X\epsilon\;;
  \end{align*}
  and $\esp[\nabla\psi(\bstheta;\xi)]=\esp[X^2](\bstheta-\bstheta^*)<\infty$. 
  Also, for a future reference, 
  $$
  \esp[|\nabla\psi(\bstheta;\xi)-\nabla\psi(\bstheta^*;\xi)|^p]\leq \esp[|X|^{2p}]|\bstheta-\bstheta^*|^p
  $$
  whenever $\esp[|X|^{2p}]<\infty$. It extends to a $d$-dimensional case of predictors in a straightforward manner. 
  }
\end{example}
\subsection{Regular variation}
In this section, we refer to \cite[Appendix~B]{kulik:soulier:2020}.  Let $\|\cdot\|$ be
any norm on $ \Rset^d$.  A measure on $\Rset^d\setminus\{0\}$ is said to be boundedly
finite if $\mu(B)<\infty$ for every set $B$ separated from $\bszero$, \ie\ there exists
$\epsilon>0$ such that $x\in B \implies \|x\|>\epsilon$.
\begin{definition}
  \label{def:vagueconvergence}
  A sequence of boundedly finite Borel measures on $\Rset^d$ $\{\mu_n\}_{n\geq1}$ converges
  vaguely to $\mu$ on $\Rset^d\setminus\{\bszero\}$, denoted $\mu_n\convvague\mu$, if for
  every Borel set $B\subset \Rset^d\setminus\{\bszero\}$ separated from zero and such that
  $\mu(\partial B)=0$, we have
  \begin{align*}
    \lim_{n\to\infty} \mu_n(B) = \mu(B)\;. 
  \end{align*}
\end{definition}
\begin{definition}
  \label{def:homogeneousmeasure}
  A measure $\mu$ on $\Rset^d\setminus\{\bszero\}$ is said to be homogeneous with index
  $\alpha$ if
  \begin{align*}
    \mu(tA) = t^{-\alpha} \mu(A) \; ,
  \end{align*}
  for all Borel sets $A$ such that $\mu(A)<\infty$.
\end{definition}

\begin{definition}
  \label{def:regvar}
  A random vector $\zeta$ is regularly varying if the sequence of measures $\{\mu_n\}_{n\geq 1}$ defined
  on $\Rset^d\setminus\{0\}$ by
  \begin{align*}
    \mu_n = \frac{\pr(n^{-1}\zeta\in \cdot)}{\pr(\|\zeta\|>n)} 
  \end{align*}
  converges vaguely to a boundedly finite measure $\mu$. 
\end{definition}
If $\zeta$ is regularly varying, then the measure $\mu$ (called the exponent measure) is homogeneous and the function
$t\mapsto\pr(\|\zeta\|>t)$ is regularly varying with the same index. Denoting $F_0$ the
cdf of $\|\zeta\|$, $F^{-1}$ its left-continuous inverse and $a_n=F_0^\leftarrow(1-1/n)$,
it also holds that
\begin{align*}
  n\pr(a_n^{-1}\zeta\in\cdot) \convvague \mu \; .
\end{align*}
\subsection{Convergence to a stable law}
\label{subsec:constable}
Our main results (\Cref{th:stable_law_convergence,th:averaging_stable_convergence}) rely
on the point process technique to prove the weak convergence to a stable law of sums of
triangular arrays of independent regularly varying random variables.

\begin{theorem}
  \label{theo:triangular_convergence}
  Let $\{X_{i,n},{1\leq i\leq m_n,n\geq 1}\}$ be a triangular array of rowwise independent
  random vectors such that $\esp[\|X_{i,n}\|]<\infty$. Define
  \begin{align*}
    S_n =\sum_{i=1}^{m_n}\left\{X_{i,n}-\esp[X_{i,n}]\right\}\;.
  \end{align*}
  Let $\nu$ be an $\alpha$-homogeneous Borel measure on $\Rset^d\setminus\{0\}$. Assume
  that
  \begin{align}
    \label{eq:convvaguetriangular}
    \sum_{i=1}^{m_n} & \ \pr\!\left({X_{i,n}}\in \cdot \right) \convvague \nu \; , \\
    \label{eq:ANSJ}
    \lim_{\epsilon\to 0} \ \limsup_{n\to\infty}   & \ \esp\left[\|X_{i,n}\|^2
    \ind{\|X_{i,n}\|\le \epsilon}\right]  = 0 \; 
  \end{align}
  as $n\to\infty$ and for every $\epsilon>0$,
  \begin{align}
    \label{eq:nullarray}
    \lim_{n\to\infty} \max_{1\le i\le m_n} & \pr(\|X_{i,n}\|>\epsilon )  = 0 \; .
  \end{align}
  Then $S_n$ converges weakly to a stable random vector $\Lambda$ with characteristic
  function given by
  \begin{align}
    \label{eq:Lambda_cf_823-d>1}
    \log\esp\left[e^{i u^\top \Lambda}\right]=\int_{\Rset^d\setminus\{0\}}
    \left\{e^{iu^\top x}-1-iu^\top x\right\}\,\nu(\rmd x) \; , \qquad u\in\Rset^d \; .
  \end{align}
\end{theorem}

\begin{remark}
  Conditions (\ref{eq:convvaguetriangular}) and (\ref{eq:nullarray}) imply that the point
  process of exceedances $\sum_{i=1}^{m_n} \delta_{X_{n,i}}$ converges weakly to a Poisson
  point process with mean measure $\nu$
  (cf. \cite[Theorem~7.1.21]{kulik:soulier:2020}). Condition (\ref{eq:ANSJ}) is known as
  asymptotic negligibility of small jumps, which is needed to apply a continuous mapping
  argument and obtain the stated convergence. See e.g. \cite[Theorem 8.3.8]{kulik:soulier:2020}.
\end{remark}


\section{Moment bound for the algorithm}
\label{sec:moment-bound}

In this section we consider the SGD algorithm with batching and present moment bounds and
convergence rates for the $L^p$ error of the SGD iterates.  For a vector norm $\|\cdot\|$
define the associated matrix norm by
\begin{align*}
  \vvvert A\vvvert = \sup_{u\in\Rset^d:\|u\|=1} \|Au\| \; .
\end{align*}
\begin{assumption}
  \label{assumption:function_H}
  Let $\Theta\subset\Rset^d$ be a compact set, $H:\Theta\to\Rset^d$ a continuously
  differentiable function. There exist constants $0<a<b$ such that for all
  $\bstheta\in\Theta$ and any eigenvalue $\sigma$ of $\nabla H(\bstheta)$,
  \begin{align}
  \label{eq:assumption_functionH_1}
  a\leq \sigma\leq b\;.
  \end{align}
\end{assumption}
\begin{assumption}[Learning rate and batch-size] 
  \label{assumption:learning_rate_general}
  \begin{enumerate}
  \item The learning rate $\{\gamma_i\}_{i\geq 1}$ is non-increasing and
    $\sum_{i=1}^{\infty}\gamma_i=\infty$.
  \item The batch size  $\{b_i\}_{i\geq 1}$ is non-decreasing.
  \item $\lim_{i\to\infty} \gamma_i/b_i=0$.
  \item The sequence $ \{(\gamma_{i}/b_{i})^{p-1}\prod_{j=1}^i(1-\gamma_j)^{-1}\}_{i\geq 1}$
    is ultimately non-decreasing.
  \end{enumerate}
\end{assumption}
\begin{assumption}
  \label{assumption:gradient_noise}
  The $\Rset^d$-valued stochastic processes
  $(\gradientnoise_i(\bstheta),\bstheta\in\Theta),i\geq 1$ are independent,
  $\esp[\|e_{i+1}(\bstheta)\|]<\infty$, $\esp[e_{i+1}(\bstheta)]=\bszero$ and there exist $p>1$
  and a constant $C_p$ such that for all $\bstheta\in\Theta$,
  \begin{align}
    \label{eq:assumption:gradient_noise_1}
    \esp[\|\gradientnoise_{i+1}(\bstheta)\|^{p}]\leq \frac{C_p}{ b_{i+1}^{p-1}}\;.
  \end{align}
 
\end{assumption}
For the gradient noise as in \Cref{def:gradient-noise-sequence}, by Marcinkiewicz–Zygmund
inequality, we have
  \begin{align*}
    \esp[\|\gradientnoise_{i+1}(\bstheta)\|^{p}]\leq \frac{C_p}{ b_{i+1}^{p-1}}
    \esp[\|\nabla\psi(\bstheta,\xi_0)\|^p] \; .
  \end{align*}
  Thus \Cref{assumption:gradient_noise} holds if
  \begin{align*}
    \sup_{\bstheta\in\Theta} \esp[\|\nabla\psi(\bstheta,\xi_0)\|^p]<\infty \; .
  \end{align*}
  The main result of this section is the following moment bound.
\begin{theorem}
  \label{th:p_framework_convergence_b>1}
  Let
  \Cref{assumption:function_H,assumption:learning_rate_general,assumption:gradient_noise}
  hold.  Then as $N\to\infty$,
  \begin{align}
    \label{eg:Lp_convergence}
    \esp[\|\bstheta_N-\bstheta^*\|^p]
    =\bigo\left(\left(\frac{\gamma_N}{b_N}\right)^{p-1}\right)\;.
  \end{align}
\end{theorem}
\paragraph{Comments:} 
\begin{enumerate}
\item 
Let $\id$ be the identity matrix of the dimension $d$. 
\Cref{assumption:function_H} 
gives a bound on 
$\vvvert\id-t\nabla H(\bstheta)\vvvert^p$
and implies uniqueness of $\bstheta^*$. See \Cref{lemma:function_H}. 
  \item \Cref{assumption:learning_rate_general} allows the learning rate to be
    non-increasing and the batch size to be constant. However, the ratio of the learning
    rate to the batch size must tend to zero.
  \item The last part of \Cref{assumption:learning_rate_general} is strictly technical.
    It is fulfilled when the sequence $\{\gamma_i\}_{i\geq 1}$ is regularly varying with
      index $\rho\in[0,1)$, $\gamma_{i+1}/\gamma_i=1+O(1/i)$ and
      $b_{i+1}/b_i=1+O(1/i)$. The last condition holds for usual regularly varying
      sequences such as $ci^\delta\log^\beta(i)$, $\beta,\delta\in\Rset$.
  \item If we consider $b_i=b\geq 1$ for $i\geq 1$
    (no batching or a constant batch size), then we are in the setup of \cite[Theorem
    3]{wang:gurbuzbalanan:shu:simsekli:erdogdu:2021} with the following result
  \begin{align*}
    \esp[\|\bstheta_N-\bstheta^*\|^p]=\bigo\left(\gamma_N^{p-1}\right)\;.
  \end{align*}
\item We can also consider a constant learning rate $\gamma_i=\gamma>0$, but in
  this case in order to make the algorithm converge the batch size must tend to $\infty$.
\item For $p=2$ without batching (finite variance case) the classical result is
  \begin{align*}
    \esp[\|\bstheta_N-\bstheta^*\|^2]=\bigo\left(\gamma_N\right)\;.
  \end{align*}
  See \cite{chung:1954}, \cite[Eq.(6)]{moulines:bach:2011},
  \cite[Eq. 2.8]{nemirovski:juditsky:lan:shapiro:2009} or \cite[Theorem
  1]{li:zhang:chen:smola:2014}.  The results in the aforementioned references are stated
  under different assumptions on the learning rate (e.g. the specific form
  $\gamma_i\sim i^{-1/2}$) and under different regularity assumptions.
\end{enumerate}
\paragraph{Complexity comparison.}
What do we achieve with batching? We will compare batching with no-batching 
by fixing a target accuracy $\varepsilon$ and solving for: (i) the required number of iterations $N(\varepsilon)$ and (ii) the required number of single-sample gradient evaluations
$n(\varepsilon) \;=\; \sum_{i=1}^{N(\varepsilon)} b_i$. 
If $\gamma_i=i^{-\rho}$ and $b_i=i^r$ then 
\begin{align*} 
\left(\mathbb{E}\left[\|\bstheta_N-\bstheta^\ast\|^p\right]\right)^{1/p} \approx \Big(\frac{\gamma_N}{b_N}\Big)^{1-1/p} = \varepsilon\;,
\end{align*}
gives
\begin{align*}
\varepsilon \approx N^{-(\rho+r)(1-1/p)}\;, \ \  N(\varepsilon)\;\approx\;\varepsilon^{-\frac{p}{(p-1)(\rho+r)}}\;, \ \ n(\varepsilon)\;\approx\;\varepsilon^{-\frac{p(r+1)}{(p-1)(\rho+r)}}\;.
\end{align*}
Since the map $r\mapsto E(r,\rho)=\frac{p(r+1)}{(p-1)(\rho+r)}$ is
  strictly decreasing with $r$ if $\rho\in(0,1)$, batching decreases complexity. 
\begin{proof}[Proof of \Cref{th:p_framework_convergence_b>1}]
  For $\mathbf{v}=(v^{(1)},v^{(2)},\dots,v^{(d)})^\top\in \Rset^d$ and $q\geq 0$, we
  define
  \begin{align*}
    v^{\langle q \rangle}=({\rm sign}(v^{(1)})|v^{(1)}|^q,\dots,{\rm sign}(v^{(d)})|v^{(d)}|^q)^\top\;.
  \end{align*}
  We use a technique similar to that of \cite{krasulina:1969} and
  \cite{wang:gurbuzbalanan:shu:simsekli:erdogdu:2021}, but for clarity we present the whole
  proof. Define the function $T_i$ by
  \begin{align*}
    T_i(\bstheta) 
    = \bstheta-\bstheta^* - \gamma_{i+1} H(\bstheta)\;.
  \end{align*}
  Then, by applying \Cref{lemma:p_norm_triangle_inequality} to our stochastic
  approximation scheme, we obtain
  \begin{align*}
    \|\bstheta_{i+1}-\bstheta^*\|^p
    & = \| T_i(\bstheta_i) - \gamma_{i+1} \gradientnoise_{i+1}(\bstheta_i) \|^p \\
    & \leq \| T_i(\bstheta_i)\|^p + 4\gamma_{i+1}^p  \| \gradientnoise_{i+1}(\bstheta_i) \|^p \\
    & \phantom{ = + } + p\gamma_{i+1} \sum_{j=1}^{d} \gradientnoise_{i+1}^{(j)}(\bstheta_i)
      |T_i^{(j)}(\bstheta_i)|^{p-1} \operatorname{sign} T_i^{(j)}(\bstheta_i)\;.
  \end{align*}
  Define $\delta_{i}=\esp[\|\bstheta_i-\bstheta^*\|^p]$. Recall that
  $\mathcal{F}_i=\sigma(\xi_j,j=1,\ldots,\sum_{k=1}^ib_k)$. Hence
  $\mathcal{F}_i=\sigma(\bstheta_j,j\leq i)$. Then by
  \Cref{assumption:gradient_noise} 
  we have
  \begin{align*}
    \esp[\gradientnoise_{i+1}^{(j)}(\bstheta_i) | T_i^{(j)}(\bstheta_i) |^{p-1}
    T_i^{(j)}(\bstheta_i) \mid \mathcal{F}_i] = 0\;,
  \end{align*}
  and by taking expectations we get
  \begin{align*}
    \delta_{i+1} \leq \esp [ \| T_i(\bstheta_i) \|^p ]
    + 4 \gamma_{i+1}^p \esp [ \| \gradientnoise_{i+1}(\bstheta_i) \|^p ]\;.
  \end{align*}
  Write
  \begin{align*}
    T_i(\bstheta)
    & =\bstheta-\bstheta^*-\gamma_{i+1}(H(\bstheta)-H(\bstheta^*)) =\int_0^1(\id-\gamma_{i+1}\nabla H((1-s)\bstheta^*+s\bstheta))
      (\bstheta-\bstheta^*)\rmd s\;.
  \end{align*}
  Next we obtain
  \begin{align*}
    &\delta_{i+1}
     \leq \esp \left[ \left\| \int_0^1(\id-\gamma_{i+1}\nabla H((1-s)\bstheta^*+s\bstheta_i))
      (\bstheta-\bstheta^*)\rmd s \right\|^p \right]
      + 4 \gamma_{i+1}^p \esp\left[ \| \gradientnoise_{i+1}(\bstheta_i) \|^p \right]\\
    & \leq \esp\left[\left(\int_0^1\|(\id-\gamma_{i+1}\nabla H((1-s)\bstheta^*+s\bstheta_i))
      (\bstheta_{i} -  \bstheta^*)\|\rmd s\right)^p\right]
      + 4 \gamma_{i+1}^p \esp [ \| \gradientnoise_{i+1}(\bstheta_i) \|^p ]\;.
 \end{align*}
 Then, by \Cref{assumption:function_H} and \Cref{lemma:function_H}, there exists
   $L$ such that for $i$ sufficiently large (for a decreasing learning rate) or for all
 $i$ (for a constant learning rate) we obtain
\begin{align*}
  \left(\int_0^1\vvvert\id-\gamma_{i+1}\nabla H((1-s)\bstheta^*+s\bstheta_i)\vvvert\rmd s\right)^p
  & \leq \sup_{\bstheta\in\Theta}\vvvert \id- \gamma_{i+1} \nabla H(\bstheta)\vvvert^p  \leq 1- L\gamma_{i+1}\;.
\end{align*}
Using \Cref{assumption:gradient_noise} and combining all the previous bounds with
$C_1=4C_0$ we get the following iterative bound
\begin{align*}
  \delta_{i+1}  &\leq (1 - L \gamma_{i+1})   \delta_i
                  +   C_1\frac{\gamma_{i+1}^{p-1}}{b_{i+1}^{p-1}} \gamma_{i+1}\;.
\end{align*}
To conclude, we use the item of \Cref{assumption:learning_rate_general} to apply
  \Cref{lem:bound-recursion-toeplitz} with $u_i=C_1(\gamma_{i}/b_{i})^{p-1}$,
  $v_i=1-L\gamma_{i+1}$ and $z_i=\gamma_i$.
\end{proof}


\section{Convergence of the algorithm}
\label{sec:convergence}

We must strengthen the assumptions of \Cref{th:p_framework_convergence_b>1}.
\begin{assumption}
  \label{assumption:regvar-GNZ}
  The random variables $\zeta_j,j\geq1$ defined by
  \begin{align*}
    \GNZ_j=\nabla\psi(\bstheta^*,\xi_j)-\esp[\nabla\psi(\bstheta^*,\xi_0)], \ \ j\geq 0
  \end{align*} 
  are \iid~regularly varying random vectors with tail index $\alpha\in(1,2)$ and exponent
  measure $\nu$.
  
\end{assumption}

\begin{assumption}
  \label{assumption:gradient_noise_new_as}
  The $\Rset^d$-valued stochastic processes
  $(\gradientnoise_i(\bstheta),\bstheta\in\Theta),i\geq 1$ are independent. For every
  $p\in(1,\alpha)$, there exists a constant $C_p>0$ such that for all $\bstheta\in\Theta$
  and $i\geq 1$,
  \begin{align*}
    \esp[\|\gradientnoise_{i}(\bstheta)-e_{i}(\bstheta^*)\|^p]
    \leq \frac{C_p}{b_i^{p-1}} \|\bstheta-\bstheta^*\|^p\;.
  \end{align*}

\end{assumption}

For the gradient noise as in \Cref{def:gradient-noise-sequence}, by Marcinkiewicz–Zygmund
inequality,
\begin{align*}
  \esp[\|\gradientnoise_{i}(\bstheta)-e_{i}(\bstheta^*)\|^p]
  \leq \frac{C}{b_i^{p-1}}
  \esp[\|\nabla\psi(\bstheta,\xi_0)-\nabla\psi(\bstheta^*,\xi_0)\|^p] \; .
\end{align*}
Thus \Cref{assumption:gradient_noise_new_as} holds if
\begin{align*}
  \esp[\|\nabla\psi(\bstheta,\xi_0)-\nabla\psi(\bstheta^*,\xi_0)\|^p]
  \leq C' \|\bstheta-\bstheta^*\|^p\; .
\end{align*}
In particular, it suffices that  \begin{align}
\label{assumption:gradient_noise_new_as-toverify}
    \|\nabla\psi(\bstheta,\xi_0)-\nabla\psi(\bstheta^*,\xi_0)\| \leq C(\xi_0)\|\bstheta-\bstheta^*\|
  \end{align}
  with $\esp[C(\xi_0)^p]<\infty$.

Note also that \Cref{assumption:regvar-GNZ,assumption:gradient_noise_new_as} imply
\Cref{assumption:gradient_noise}.

\begin{assumption} 
  \label{assumption:function_H_taylor_exp}
  For every $q\in(1,\alpha)$, there exists a constant $C_q$ such that, for all
  $\bstheta\in\Theta$,
  \begin{align*}
    \|H(\bstheta)-H(\bstheta^*)-\nabla H(\bstheta^*)\left(\bstheta-\bstheta^*\right)\|
    \leq C_q\|\bstheta-\bstheta^*\|^q\;.
  \end{align*}
\end{assumption}

\begin{assumption}
  \label{assumption:learning-rate-rv}
  The learning rate $\{\gamma_i\}_{i\geq 1}$ is non-increasing and regularly varying with
  index $-\rho$, $\rho\in [0,1)$ and the batch size $\{b_i\}_{i\geq1}$ is non-decreasing
  and regularly varying with index $r\geq0$, $\lim_{i\to\infty}\gamma_i/b_i=0$ and the
  sequence $ \{(\gamma_{i}/b_{i})^{p-1}\prod_{j=1}^i(1-\gamma_j)^{-1}\}_{i\geq1}$ is
  ultimately non-decreasing.
\end{assumption}
This assumption is a mild strengthening of \Cref{assumption:learning_rate_general}. We
only assume additionally that the sequence $\{b_i\}_{i\geq 1}$ is regularly varying to make certain
computations simpler.

Let $F_0$ be the cdf of $\|\GNZ_0\|$ and $\quantilefunction{F}[0]$ its left-continuous
inverse. Under assumption \Cref{assumption:regvar-GNZ}, the map
$x\mapsto\quantilefunction{F}[0](1-1/x)$ is regularly varying at infinity with index
$1/\alpha$.

Define the sequence $\{w_i\}_{i\geq 0}$ by $w_0=1$ and
\begin{align*}
  w_i=\frac{b_i}{\gamma_i\quantilefunction{F}[0]\left(1-\frac{\gamma_i}{b_i}\right)}\;.
\end{align*}
Under \Cref{assumption:regvar-GNZ,assumption:learning-rate-rv}, $w_i$ is regularly
varying with index $(1-1/\alpha)(r+\rho)$.  To deal with some remainder terms, we need to
impose a kind of regularity condition which is satisfied for most sequences $\{\gamma_i\}_{i\geq 1}$ and
$\{b_i\}_{i\geq 1}$ and regularly varying function $\overline{F}_0$ (see the discussion below \Cref{th:stable_law_convergence}): 
\begin{align}
  \label{eq:diffw}
  w_{i+1}-w_i = \bigo\left(\frac{w_i}i\right) \; .
\end{align}

\begin{theorem}
  \label{th:stable_law_convergence}
  Let
  \Cref{assumption:function_H,
    assumption:regvar-GNZ,assumption:gradient_noise_new_as,%
    assumption:function_H_taylor_exp,%
    assumption:learning-rate-rv} hold with $\gamma_1 L<1$. Assume also that
  \Cref{eq:diffw} holds. Then, as $N\to\infty$,
  \begin{align*}
    w_N(\bstheta_N-\bstheta^*)\Rightarrow  Z_{\infty} \; , 
  \end{align*}
  with $Z_{\infty}$ a stable random vector with characteristic function given by
  \begin{subequations}
    \begin{align}
      \label{eq:chf}
      \esp\left[\exp(iu^\top Z_\infty)\right]
      &= \exp\left(\int_{\Rset^d\setminus\{0\}}\Big(e^{iu^\top x}-1-iu^\top x\Big)
        \,\widetilde{\nu}(\rmd x)\right), \ \ 
      \\
      \label{eq:nutilde-measure}
      \widetilde{\nu} & =\int_{t=0}^{\infty}\nu(e^{t\,\nabla H(\bstheta^*)}\cdot)\,\rmd t \; ,
    \end{align}
  \end{subequations}
  and $\nu$ is the exponent measure of  $\GNZ_0$.

\end{theorem}

\paragraph{Comments:}
\begin{enumerate}
\item 
  If the batch size is constant, we recover the rate in \cite[Thm
  4.1]{blanchet:2024}.
\item The distribution of $Z_{\infty}$ is the stationary distribution of an
  Ornstein-Uhlenbeck process driven by a L\'{e}vy stable process $L$ such that the
  distribution of $L(1)$ has the characteristic function
  \begin{align*}
    \exp    \int_{\mathbb{R}^d\setminus\{0\}}
    \left\{e^{iu^\top x}-1-iu^\top x \right\} \nu(\rmd x)
    \; .
  \end{align*}
 That means
 \begin{align*}
   Z_{\infty}\eqdistr \int_0^\infty e^{-s A}\rmd L(s)\;.
 \end{align*}
 The characteristic function of the limiting random variable is given in
 \cite{blanchet:2024} in the following form:
 \begin{align}
   \label{eq:chf-2}
   \exp\left(iu^\top\widetilde\gamma +
   \int_{\mathbb{R}^d\setminus\{0\}}
   \Big(e^{iu^\top x}-1-iu^\top x 1\{\|x\|\leq 1\}\Big)\,\widetilde{\nu}(\rmd x) \right)
 \end{align}
 with
 \begin{align}
   \label{eq:drift}
   \tilde{\gamma}
   & = -A^{-1}\int_{\|x\|>1} x\,\nu(\rmd x) \nonumber\\
   & \quad + \int_{ 0}^\infty e^{-A t} x
     \left(\ind{\|\exp(-A t)x\|\le 1} - \ind{\|x\|\le 1)}\right)
     \nu(\rmd x)\,\rmd t \; .
 \end{align}
 We show in \Cref{lem:drift} that
 $\widetilde{\gamma}=\int_{\|x\|>1} x\widetilde{\nu}(\rmd x)$, which implies that the
 characteristic functions in \Cref{eq:chf} and \Cref{eq:chf-2} are equal.

\item \Cref{assumption:regvar-GNZ} is very similar to \cite[Assumption
  2.5]{blanchet:2024} {but we only assume the} regular variation of
  $\nabla\psi(\bstheta^*,\xi_j)$.
\item \Cref{assumption:gradient_noise_new_as} is in the spirit of \cite[Assumption
  2.7]{blanchet:2024}, with batching added. Note however that while the latter article
  requires an almost sure bound, we require a moment bound only. This has some
  consequences. For example, when considering linear regression \cite[Assumption
  2.7]{blanchet:2024} implies that the predictors have to be bounded almost surely. {In our case, 
  \eqref{assumption:gradient_noise_new_as-toverify} holds whenever the predictors have finite moment of order $2p$. See \Cref{xmpl:linear-regression}.} 
\item \Cref{assumption:function_H_taylor_exp} is \cite[Assumption
  2.8]{blanchet:2024}. It holds if the function $H$ is $C^2$, but could not be used
    with $q=2$ under the present heavy tail framework. 
\item We comment on  \eqref{eq:diffw}. It is satisfied for a variety of regularly varying distributions $F$, and regularly varying sequences $\{b_i\}_{i\geq 1}$ and $\{\gamma_i\}_{i\geq 1}$. To give a concrete example,  if 
    \begin{align*}
    \overline{F_0}(x)=1-F_0(x) & = \frac{1}{x^\alpha \ell_0(x^\alpha)}
    \end{align*}
    then 
    \begin{align*}
    \quantilefunction{F}[0](1-1/x) & = \left(x\ell_0^\#(x)\right)^{1/\alpha} \; ,
  \end{align*}
  where $\ell_0^\#$ is a de Bruijn conjugate of $\ell_0$. 
  Choose for example 
$\ell_0(x)=\log^{\kappa}(x)$, $\kappa\in\Rset$. Then 
$\ell^{\#}(x)=1/\log^\kappa(x)$ and 
$$
  w_i=
  \left(\frac{b_i}{\gamma_i}\right)^{1-1/\alpha}\frac{1}{(\ell^{\#}(b_i/\gamma_i))^{1/\alpha}}=
  \left(\frac{b_i}{\gamma_i}\right)^{1-1/\alpha}
  (\log(b_i)-\log(\gamma_i))^{\kappa/\alpha}\;.
    $$
  Choose $b_i=i^r$, $r\geq 0$ and $\gamma_i=i^{-\rho}$, $\rho\in (0,1)$. Then we can write
 $$
 w_i=C i^\beta\log^{\eta}(i)
 $$
 with $C,\beta,\eta>0$. Then 
 $$
 \lim_{i\to\infty}i\frac{w_{i+1}-w_i}{w_i}=\beta
 $$
 and hence  \eqref{eq:diffw} is satisfied.

\item We exclude the case $\rho=1$, which requires separate analysis, because of some bias
  terms which are non negligible in that case. See \citet[Corollary 4.3]{blanchet:2024}.
\end{enumerate}
\paragraph{Complexity comparison.}
Similarly to the case of $L^p$ bounds, we assume that the length of the confidence interval, $1/w_N$, is $\varepsilon$:
\begin{align*}
1/w_N \approx \Big(\frac{\gamma_N}{b_N}\Big)^{1-1/\alpha} = \varepsilon\;.
\end{align*}
Then, by the same computations as in case of $L^p$ bounds, 
$n(\varepsilon)=\varepsilon^{-\frac{\alpha(r+1)}{(\alpha-1)(\rho+r)}}$; 
the batching decreases complexity. \\

We now prove \Cref{th:stable_law_convergence}.  Our proof is different from that of
\citet{blanchet:2024}, which does not consider batching and relies on an embedding
technique borrowed from the finite variance case (see \cite{pelletier:1998}). The
  most delicate part of the proof of \cite{blanchet:2024} is the proof of the tightness of
  $w_N(\bstheta_N-\bstheta^*)$. We actually found it easier to directly prove the
  convergence than the tightness.  To this end, we approximate
  $w_N(\bstheta_N-\bstheta^*)$ by a weighted sum of independent, regularly varying random
  variables and apply \Cref{theo:triangular_convergence}. The most technical details of
the proof are postponed to \Cref{sec:appendix-technical-details-convergence}.
\begin{proof}[Proof of \Cref{th:stable_law_convergence}]
  Recall that by definition, we have 
  \begin{align*}
    \bstheta_{i+1}=\bstheta_i-\gamma_{i+1}H(\bstheta_i)
    -\gamma_{i+1}\gradientnoise_{i+1}(\bstheta_i)\;.
  \end{align*}
  Let $A=\nabla H(\bstheta^*)$. A first order Taylor expansion of $H$ around
  $\bstheta^*$ yields
  \begin{align}
    \nonumber
    &\bstheta_{i+1}-\bstheta^*
    =(\id-\gamma_{i+1}A)(\bstheta_i-\bstheta^*)
      - \gamma_{i+1}\gradientnoise_{i+1}(\bstheta_i)+\gamma_{i+1}R^{(1)}_{i+1} \\
    \label{eq:two_rests}
    & = (\id-\gamma_{i+1}A)(\bstheta_i-\bstheta^*)
      -\gamma_{i+1}\gradientnoise_{i+1}(\bstheta^*)+\gamma_{i+1}R^{(1)}_{i+1}+\gamma_{i+1}R_{i+1}^{(2)}\;,
  \end{align}
  where
  \begin{subequations}
    \begin{align}
      \label{eq:Rn1}
      R_{i+1}^{(1)} & =  H(\bstheta_i)-H(\bstheta^*)
                      -\nabla H(\bstheta^*)\left(\bstheta_i-\bstheta^*\right) \; , \\
      \label{eq:Rn2}
      R_{i+1}^{(2)} & = \gradientnoise_{i+1}(\bstheta^*)-\gradientnoise_{i+1}(\bstheta_i)\;,
    \end{align}
  \end{subequations}
  Next, we define
  \begin{align}
    \label{eq:X_n_def}
    X_N=w_N(\bstheta_N-\bstheta^*)\;.
  \end{align}
  Therefore, by \eqref{eq:two_rests}
  \begin{align*}
    X_{N+1} = \frac{w_{N+1}}{w_N}(\id-\gamma_{N+1}A)X_N
    & - w_{N+1}\gamma_{N+1}\gradientnoise_{N+1}(\bstheta^*)   + w_{N+1}\gamma_{N+1}R^{(1)}_{N+1}+w_{N+1}\gamma_{N+1}R_{N+1}^{(2)}\;.
  \end{align*}
  Let $\widetilde{X}_N^\dag$ to be the solution of the modified recursion 
  \begin{align*}
    \widetilde{X}_{0}^\dag & = \bstheta_0-\bstheta^*,\\
    \widetilde{X}_{N+1}^\dag & = (\id-\gamma_{N+1}A)\widetilde{X}_N^\dag
                               -\gamma_{N+1}w_{N+1}\gradientnoise_{N+1}(\bstheta^*)\;.
  \end{align*}
  Consequently,
  \begin{align}
    \label{eq:X_n_tilda_def-dag}
    \widetilde{X}_{N}^\dag
    = \sum_{i=1}^{N}\left(\prod_{k=i+1}^{N}(\id-A\gamma_{k})\right)  w_i\gamma_i
    \gradientnoise_i(\bstheta^*)
    + \prod_{i=1}^{N}(\id-A\gamma_{i})(\bstheta_0-\bstheta^*)\;.
  \end{align}
  Define finally
  \begin{align}
    \label{eq:X_n_tilda_def}
    \widetilde{X}_N=\sum_{i=1}^{N}\left(\prod_{k=i+1}^{N}(\id-A\gamma_{k})\right)
    w_i\gamma_i 
    \gradientnoise_i(\bstheta^*)\;.
  \end{align}
  By \Cref{assumption:function_H} and \Cref{lemma:function_H} and since we assume that the sequence $\{\gamma_i\}_{i\geq 1}$ is
  non-increasing and $\gamma_1L<1$, we have
    \begin{align*}
      \left\|    \widetilde{X}_{N}^\dag - \widetilde{X}_N\right\|^p
      & \leq \prod_{i=1}^{N}\| (\id-\gamma_{i}A)(\bstheta_0-\bstheta^*)\|^p \\
      & \leq \|\bstheta_0-\bstheta^*\|^p\prod_{i=1}^{N}\vvvert (\id-\gamma_{j}A)\vvvert^p  
        \leq \|\bstheta_0-\bstheta^*\|^p\prod_{i=1}^{N}(1-L\gamma_{i}) \; .
    \end{align*}  
  Then,  using $\log(1-x)\leq -x$ for $x\in(0,1)$ and since $\gamma_i$ is regularly
  varying with index $\rho<1$, we obtain
  \begin{align*}
    \log\left(\prod_{i=1}^{N}(1-L\gamma_{i})\right)
    = \sum_{i=1}^{N}\log\left(1-L\gamma_{i}\right)
    \leq -L\sum_{i=1}^{N}\gamma_i\to\-\infty, \text{ as } N\to\infty\;.
  \end{align*}
  Thus, $\lim_{N\to\infty}\prod_{i=1}^{N}(1-L\gamma_{i})=0$ and
  $ \widetilde{X}_{N}^\dag - \widetilde{X}_N=o_P(1)$.
  \Cref{lem:X_n_and_X_n_tilda_convergence} shows that
  $X_N-\widetilde{X}_N=o_P(1)$, thus there only remains to prove that $\widetilde{X}_N$
  converges to the stated limit.
 
  We write $\widetilde{X}_N$ as a weighted sum of the random variables $\GNZ_j$ from
  \Cref{assumption:regvar-GNZ}. Set
  \begin{subequations}
    \begin{align}
      \pi_0 & = 0 \; ,  \ \       \pi_i  = \sum_{j=1}^ib_j \; , \ \ i\geq 1 \; ,
              \label{eq:debatching_1a} \\
      k_i & = \sup\{k:\pi_k< i\} \; . \label{eq:debatching_1b}
    \end{align}
  \end{subequations}
  This entails that for $q\geq1$,
  \begin{align*}
    k_i=q \iff \pi_{q-1}< i\leq \pi_q \; .
  \end{align*}
  Define also the deterministic matrices
  \begin{subequations}
    \begin{align}
      A_{N,i} & = \prod_{k=i+1}^{N}(\id-\gamma_{k}A), \ \ i=1,\ldots,N-1, \ \
                A_{N,N}=I\;, \label{eq:w_n_k_definition}
      \\
      W_{N,i} &  = \frac{1}{\quantilefunction{F}[0]\left(1-\frac{\gamma_i}{b_i}\right)}A_{N,i}, \ \
                i=1,\ldots,N\;.\label{eq:W_n_k_definition}
    \end{align}
  \end{subequations}
  Then,
  \begin{align}
    \label{eq:Xtilde}
    \widetilde{X}_N
    = \sum_{i=1}^NW_{N,i}b_i\gradientnoise_i(\bstheta^*)
    = \sum_{j=1}^{\pi_N}W_{N,k_j}\GNZ_{j}\;.
  \end{align}
  Hence, $\widetilde{X}_N$ is a weighted sum of $\pi_N$ independent regularly varying
  random variables. Its convergence will be proved by applying
  \Cref{theo:triangular_convergence}. This means that we must check conditions
  \Cref{eq:convvaguetriangular}, \Cref{eq:nullarray} and \Cref{eq:ANSJ} which in the
  present context become
  \begin{subequations}
    \begin{align}
      \label{eq:meanmeasure-convergence-SGD}
      \sum_{j=1}^{\pi_N}\pr(W_{N,k_j}\GNZ_0\in  \cdot)
      & = 
        \sum_{i=1}^{N}b_i\pr(W_{N,i}\GNZ_0\in  \cdot) \convvague \widetilde{\nu} \; , \\
      \label{eq:nullarray-convergence-SGD}
      \lim_{N\to\infty}\max_{1\leq i\leq N} & \pr(\|W_{N,i}\GNZ_0\|>\epsilon)= 0 \; , \\
      \label{eq:ANSJ-convergence-SGD}
      \lim_{\epsilon\to 0}\limsup_{N\to\infty}
      & \sum_{j=1}^{\pi_N}
        \esp[\|W_{N,k_j}\GNZ_j\|^2\ind{\|W_{N,k_j}\GNZ_j\|\leq\epsilon}]= 0 \; , 
    \end{align}
  \end{subequations}
  where $\widetilde{\nu}$ is defined in \eqref{eq:nutilde-measure}. 
\end{proof}
The main task is to prove \Cref{eq:meanmeasure-convergence-SGD}, which we do now.
\begin{proof}[Proof of \Cref{eq:meanmeasure-convergence-SGD}]
  Define the measures
  \begin{align*}
    \nu_i=\frac{\pr(\GNZ_0\in \quantilefunction{F}[0](1-\gamma_i/b_i)\cdot)}
    {\pr(\|\GNZ_0\|>\quantilefunction{F}[0](1-\gamma_i/b_i))}\;.
  \end{align*}
  Define also
  \begin{align}
    \label{eq:def-kappai}
    \kappa_i = b_i      \pr( \|\GNZ_0\| > \quantilefunction{F}[0](1-\gamma_i/b_i) ) \; .
  \end{align}
  If $F_0$ is continuous, then $\kappa_i=\gamma_i$. We do not make this assumption.
   Then for every set $S$
  separated from zero, we must prove that 
  \begin{align}
    \label{eq:mean-measure-d>1-one}
    \mu_N(S) & = \sum_{i=1}^{N}b_i \pr(W_{N,i}\GNZ_0\in S) 
               = \sum_{i=1}^{N}\kappa_i\nu_i(W_{N,i}^{-1} S)  \to  \widetilde{\nu}(S) \; .
  \end{align}
  The key to proving this convergence is to realize that only the last few terms of the
  sum contribute to the limit.  Fix $K>0$ and $K_N=K/\gamma_N$. We will prove that
  \begin{subequations}
    \begin{align}
      \label{eq:lim_lim_sup_K}
      \lim_{K\to\infty}\limsup_{N\to\infty}
      & \sum_{i=1}^{N-K_N}b_i\pr(W_{N,i}\GNZ_0\in S)=0 \; ,  \\
      \label{eq:mean-measure>K}
      \lim_{N\to\infty}
      & \sum_{i=N-K_N}^{N}\kappa_i\nu_i(W_{N,i}^{-1}S)=\int_0^K \nu(e^{tA}S)\rmd t\;.
    \end{align}
  \end{subequations}
  Combining \eqref{eq:lim_lim_sup_K} and \eqref{eq:mean-measure>K} and letting
  $K\to\infty$ yields \eqref{eq:mean-measure-d>1-one}.

  Write
  \begin{align*}
    r_n(S) &  = \sum_{i=1}^{N-K_N}b_i\pr(W_{N,i}\GNZ_0\in S) \; .
  \end{align*}
  and for $\sigma>0$, define
  \begin{align}
    \label{def:wsigma}
    w_{N,i}(\sigma)=\prod_{k=i+1}^{N}(1-\sigma\gamma_k) \; .
  \end{align}
  Since $S$ is separated from $0$ and $A$ is positive definite (by
  \Cref{assumption:function_H}), there exist $\sigma,\epsilon>0$ such that
  \begin{align*}
   r_n(S) 
    &\le \sum_{i=1}^{N-K_N} b_i
      \pr( w_{N,i}(\sigma)\|\GNZ_0\| > \quantilefunction{F}[0](1-\gamma_i/b_i) \epsilon) \; .
  \end{align*}
  For each fixed $i$, $w_{N,i}(\sigma)\to0$, thus for any $i_0\geq1$,
  \begin{align*}
    \lim_{N\to\infty}  \sum_{i=1}^{i_0} b_i
      \pr( w_{N,i}(\sigma)\|\GNZ_0\| > \quantilefunction{F}[0](1-\gamma_i/b_i) \epsilon) = 0 \; .
  \end{align*}
  Thus we can truncate the sum $r_n(S)$ at any arbitrary $i_0$.   By
  regular variation, there exists a slowly varying function $\ell_0$ such that
  \begin{align*}
    \overline{F_0}(x)=1-F_0(x)  = \frac{1}{x^\alpha \ell_0(x^\alpha)} \; , \ \ 
    \quantilefunction{F}[0](1-1/x)  = \left(x\ell_0^\#(x)\right)^{1/\alpha} \; ,
  \end{align*}
  where $\ell_0^\#$ is a de Bruijn conjugate of $\ell_0$, which satisfies
  \begin{align}
    \label{eq:carac-debruijn}
    \lim_{x\to\infty} \ell_0(x) \ell_0^\#(x \ell_0(x))=
    \lim_{x\to\infty} \ell_0^\#(x) \ell_0(x \ell_0^\#(x))= 1 \; .
  \end{align}
  Cf. \citet[Proposition~1.5.15]{bingham:goldie:teugels:1989}. Thus, as $i\to\infty$,
  \begin{align*}
    \frac{\kappa_i}{\gamma_i}
    &  = \frac{b_i}{\gamma_i}    \overline{F_0}(\quantilefunction{F}[0](1-\gamma_i/b_i) )  = \frac1{\ell_0^\#(b_i/\gamma_i)\ell_0((b_i/\gamma_i)\ell_0^\#(b_i/\gamma_i)} \to 1 \;  .
  \end{align*}  
  Fix $i_0\geq1$ and set
  \begin{align*}
    \tilde{r}_n(S) 
    & = \sum_{i=i_0+1}^{N-K_N} b_i
      \pr( w_{N,i}(\sigma)\|\GNZ_0\| > \quantilefunction{F}[0](1-\gamma_i/b_i) \epsilon) \\
    & =  \sum_{i=i_0+1}^{N-K_N} \kappa_i
      \frac{\pr( w_{N,i}(\sigma)\|\GNZ_0\| > \quantilefunction{F}[0](1-\gamma_i/b_i) \epsilon)}
      {\pr( \|\GNZ_0\| > \quantilefunction{F}[0](1-\gamma_i/b_i) )} \; .
  \end{align*}  
  Fix $\eta\in(0,\alpha)$ and a constant $C>1$. Since $b_i/\gamma_i\to\infty$ and
  $w_{N,i}(\sigma)\leq1$, by Potter's theorem
  (cf. \citet[Theorem~1.5.6]{bingham:goldie:teugels:1989}), we can choose $i_0$ such that
  for all $i\geq i_0$,
  \begin{align*}
    \frac{\pr(w_{N,i}(\sigma)\|\GNZ_0\|>\quantilefunction{F}[0](1-\gamma_i/b_i)\epsilon)}
    {\pr(\|\GNZ_0\|>\quantilefunction{F}[0](1-\gamma_i/b_i))} \leq C w_{N,i}^{\alpha-\eta}(\sigma) \; ,
  \end{align*}
  where the constant $C$ depends on $\epsilon$.  Thus
  \begin{align*}
    \tilde{r}_n(S) \leq C \sum_{i=1}^{N-K_N} \kappa_i  w_{N,i}^{\alpha-\eta}(\sigma) \; .
  \end{align*}
  For $i\le N-K_N$ we write
  \begin{align*}
    w_{N,i}(\sigma)=\prod_{k=i+1}^{N}(1-\sigma\gamma_k)=w_{N-K_N,i}(\sigma)\,w_{N,N-K_N+1}(\sigma) \; , 
  \end{align*}
  and therefore 
  \begin{align*}
    \tilde{r}_n(S) \leq C 
    w_{N,N-K_N+1}^{\alpha-\eta}(\sigma) \sum_{i=1}^{N-K_N} {\kappa_i} w_{N-K_N,i}^{\alpha-\eta}(\sigma) \; .
  \end{align*}
  By \Cref{lem:bound-recursion-toeplitz} (applied with $u_i=1$,
    $v_i=(1-\sigma\gamma_i)^{\alpha-\eta}$ and $z_i=\kappa_i$), we know that the sum
  above is uniformly bounded.  Since $\log(1-x)\le -x$ for $x\geq0$, we have
  \begin{align*}
    w_{N,N-K_N}^{\alpha-\eta}(\sigma)=\prod_{k=N-K_N+1}^N(1-\sigma\gamma_k)^{\alpha-\eta}
    \le\exp\Big(-(\alpha-\eta)\sigma\sum_{k=N-K_N+1}^N\gamma_k\Big)\;.
  \end{align*}
  By the regular variation of $\{\gamma_i\}_{i\geq 1}$, there exists $c>0$ such that for all $N$
  large enough
  \begin{align*}
    \sum_{j=N-K_N+1}^N\gamma_j\ge c\,K_N\gamma_N\ge cK\;.
  \end{align*}
  and therefore
  \begin{align*}
    w_{N,N-K_N}^{\alpha-\eta}(\sigma)\le \rme^{-c'K}
  \end{align*}
  for some $c'>0$. Combining the above estimates yields, for some constant~$C'$, 
  \begin{align*}
    \tilde{r}_n(S) \leq C' \rme^{-c'K}.
  \end{align*}
  Letting $K\to\infty$ concludes the proof of \Cref{eq:lim_lim_sup_K}.  We now prove
  \Cref{eq:mean-measure>K}. Let $G$ be the {piecewise linear} function such that
  \begin{align*}
    G(0)=0,\qquad G(N)=\sum_{i=1}^{N}\delta_i\;,
  \end{align*}
  and $\quantilefunction{G}$ be its inverse defined by
  \begin{align*}
    G^{-1}(u)=\inf\{x:G(x)\ge u\}\;.
  \end{align*}
  Both $G$ and $G^{-1}$ are continuous and
  $G^{-1}\!\left(\sum_{i=1}^{N}\delta_i\right)=N$. Let $g$ be the derivative of $G$, i.e.
  \begin{align*}
    g(t)=\delta_N\ \text{ if } N-1<t\le N\;.
  \end{align*}
  Then, $G^{-1}$ is also piecewise differentiable and
  \begin{align*}
    (G^{-1})'(t)=\frac{1}{g(N)}\ \text{ if } G(N-1)<t<G(N)\;.
  \end{align*}
  By regular variation of the sequence $\{\gamma_i\}_{i\geq 1}$ hence of $\{\delta_i\}_{i\geq 1}$, the
  function $G$ is regularly varying with index $1-\rho$ and $G^{-1}$ is regularly varying
  with index $1/(1-\rho)$.

  Define $B_N(t)=W_{N,G^{-1}(t)}$, i.e. $B_N(t)=W_{N,i}$ if $G(i-1)<t\le G(i)$. Then
  \begin{align*}
    \sum_{i=N-K_N}^{N}\delta_i \nu_i(W_{N,i}^{-1}S)
    & = \int_{G(N-K_N)}^{G(N)} \nu_{G^{-1}(t)}(B_N(t)S)\,\rmd t + {\bigo(\delta_{N-K_N})} \\
    & = \int_{0}^{G(N)-G(N-K_N)} \nu_{G^{-1}(t)}\!\big(B_N(G(N)-t)S\big)\,\rmd t + {\bigo(\delta_{N-K_N})}\;,
  \end{align*}
  where ${\bigo(\delta_{N-K_N})}$ is the correction for the lower bound.  We will prove
  that
  \begin{align}
    \label{eq:G-G_bounded_uniformly}
    \lim_{N\to\infty}G(N)-G(N-K_N)=K\;,
  \end{align}
  and uniformly for $t\le G(N)-G(N-K_N)$,
  \begin{subequations}
    \begin{align}
      \label{eq:conv_to_t}
      \lim_{N\to\infty}\gamma_N\big(N-G^{-1}(G(N)-t)\big)=t, \\
      \lim_{N\to\infty}B_N(G(N)-t)=e^{-tA}\;. \label{eq:conv_to_e^A}
    \end{align}
  \end{subequations}
  If the above are proven, then the extended continuous mapping theorem (cf. \cite[Theorem
  5.5]{billingsley:1968}) gives
  \begin{align*}
    \lim_{N\to\infty}\int_{0}^{G(N)-G(N-K_N)} \nu_{G^{-1}(t)}\!\big(B_N(G(N)-t)S\big)\,\rmd t
    =\int_{0}^{K}\nu(e^{tA}S)\,\rmd t\;.
  \end{align*}
  Thus, we now prove \eqref{eq:G-G_bounded_uniformly} and
  \eqref{eq:conv_to_t}-\eqref{eq:conv_to_e^A}. For \eqref{eq:G-G_bounded_uniformly}, we
  have
  \begin{align*}
    G(N)-G(N-K_N)
    & = \int_{N-K_N}^{N} g(t)\,\rmd t
      = Ng(N)\int_{1-K_N/N}^{1}\frac{g(Ns)}{g(N)}\,\rmd s\;.
  \end{align*}
  Since $K_N/N\to 0$, by regular variation and the uniform convergence theorem
  (\cite[Theorem 1.5.2]{bingham:goldie:teugels:1989}), we have
  \begin{align*}
    G(N)-G(N-K_N) & \sim Ng(N)\int_{1-K_N/N}^{1}s^{-\rho}\,\rmd s \\
                  & = Ng(N)\big(1-(1-K_N/N)^{1-\rho}\big)/(1-\rho) \\
                  & \sim Ng(N)K_N/N = K\;.
  \end{align*}
  This proves \eqref{eq:G-G_bounded_uniformly} and consequently  $G(N)-G(N-K_N)$ is bounded.

  Next for $t\leq G(N)-G(N-K_N)$, applying again the uniform convergence theorem,
  \begin{align*}
    &\gamma_N\big(N-G^{-1}(G(N)-t)\big)
     = \gamma_N\int_{G(N)-t}^{G(N)}\frac{1}{g(G^{-1}(s))}\,\rmd s \\
    & = \frac{\gamma_N G(N)}{g(N)}\int_{1-t/G(N)}^{1}
      \underbrace{\frac{g(G^{-1}(G(N)))}{g(G^{-1}(G(N)s))}}_{{\to 1,\ \mathrm{uniformly}}}\,\rmd s   \sim \frac{\gamma_N G(N)}{g(N)}\frac{t}{G(N)} = t\;.
  \end{align*}
  The convergence \eqref{eq:conv_to_t} is uniform \wrt~$t$ on compact subsets of
  $(0,\infty)$.

  Finally,
  \begin{align*}
    B_{G^{-1}(G(N)-t)}
    & = \prod_{i=G^{-1}(G(N)-t)+1}^{N}(\id-\gamma_i A) \\
    & {\sim (\id-\gamma_N A)^{N-G^{-1}(G(N)-t)}}  \sim e^{-\gamma_N(N-G^{-1}(G(N)-t))A}\to e^{-tA}\;,
  \end{align*}
  uniformly on $t\le G(N)-G(N-K_N)$. This proves \eqref{eq:conv_to_e^A} and concludes the
  proof of \Cref{eq:mean-measure>K} and \Cref{eq:meanmeasure-convergence-SGD}.
\end{proof}

\begin{proof}[Proof of \Cref{eq:nullarray-convergence-SGD}]
  Since $\gamma_i/b_i\to0$, for any fixed $\eta>0$ there exists $i_0$ such that
  \begin{align*}
    \max_{i>i_0} \overline{F_0}\left(\epsilon \quantilefunction{F}[0](1-\gamma_i/b_i)\right) \leq \eta  \; .
  \end{align*}
  Set $C=\epsilon \quantilefunction{F}[0](1-\gamma_{i_0}/b_{i_0})$. The sequence $\{\gamma_i/b_i\}_{i\geq 1}$ is
  non-increasing, $\vvvert A_{N,i}\vvvert\leq 1$, hence
  \begin{align*}
    &\max_{1\leq i\leq N}\pr(\|W_{N,i}\GNZ_0\|>\epsilon)
     = \max_{1\leq i\leq N}\overline{F_0}
      \left(\frac{\epsilon \quantilefunction{F}[0]\left(1-\frac{\gamma_i}{b_i}\right)}
      {\vvvert A_{N,i}\vvvert}\right)    \\
    & \leq \max_{1\leq i\leq i_0}\overline{F_0}
      \left(\frac{C}{\vvvert A_{N,i}\vvvert}\right)
      +\max_{i>i_0}\overline{F_0}
      \left(\epsilon \quantilefunction{F}[0]\left(1-\frac{\gamma_i}{b_i}\right)\right)  \leq \max_{1\leq i\leq i_0}\overline{F_0}
      \left(\frac{C}{\vvvert A_{N,i}\vvvert}\right) + \eta \; .
  \end{align*}
  Since $\lim_{N\to\infty}\vvvert A_{N,i}\vvvert=0$ for each fixed $i$, we have
  \begin{align*}
    \limsup_{N\to\infty} \max_{1\leq i\leq N}\pr(\|W_{N,i}\GNZ_0\|>\epsilon) \leq \eta \; .
  \end{align*}
  Since $\eta$ is arbitrary, this concludes the proof.
\end{proof}

\begin{proof}[Proof of \Cref{eq:ANSJ-convergence-SGD}]
  Let $\sigma$ be the largest eigenvalue of $A$ and let $w_{N,i}(\sigma)$ be as
  in~\Cref{def:wsigma}. 
  Then
  \begin{align*}
    \sum_{j=1}^{\pi_N}
    & \esp\left[\|W_{N,k_j}\GNZ_0\|^2 \ind{\|W_{N,k_j}\GNZ_0\|\leq \epsilon}\right] \\
    & = \sum_{i=1}^N b_i\esp\left[\|W_{N,i}\GNZ_0\|^2
      \ind{\|W_{N,i}\GNZ_0\|\leq\epsilon}\right] \\
    & \leq \sum_{i=1}^N \frac{b_iw_{N,i}^2(\sigma)}{(\quantilefunction{F}[0](1-\gamma_i/b_i))^2}
      \esp\left[\|\GNZ_0\|^2
      \ind{w_{N,i}(\sigma)\|\GNZ_0\|\leq\epsilon \quantilefunction{F}[0](1-\gamma_i/b_i)}\right] \\
    & \leq C\epsilon^{2} \sum_{i=1}^N b_i 
      \pr\left(w_{N,i}(\sigma)\|\GNZ_0\|>\epsilon \quantilefunction{F}[0](1-\gamma_i/b_i)\right)  \; , 
  \end{align*}
  where the last bound is obtained by Karamata theorem,
  cf.~\citet[Proposition~1.4.6]{kulik:soulier:2020}. Next, recalling $\kappa_i$ from
  \Cref{eq:def-kappai} and using Potter's theorem as in the proof of
  \Cref{eq:meanmeasure-convergence-SGD}, we obtain, for $C>0$ and $\eta\in(0,\alpha)$ and
  $i_0$ large enough,
  \begin{align*}
    \sum_{j=1}^{\pi_N}
    & \esp\left[\|W_{N,k_j}\GNZ_0\|^2 \ind{\|W_{N,k_j}\GNZ_0\|\leq \epsilon}\right] \\
    & \leq C\epsilon^{2} \sum_{i=1}^{i_0} b_i 
      \pr\left(w_{N,i}(\sigma)\|\GNZ_0\|>\epsilon \quantilefunction{F}[0](1-\gamma_i/b_i)\right)  \\
    & \phantom{\leq +}
      + C\epsilon^{2} \sum_{i=i_0+1}^N \kappa_i
      \frac{\pr\left(w_{N,i}(\sigma)\|\GNZ_0\|>\epsilon \quantilefunction{F}[0](1-\gamma_i/b_i)\right)}
      {\pr\left(\|\GNZ_0\|> \quantilefunction{F}[0](1-\gamma_i/b_i)\right)}  \\
    & \leq C\epsilon^{2} \sum_{i=1}^{i_0} b_i 
      \pr\left(w_{N,i}(\sigma)\|\GNZ_0\|>\epsilon \quantilefunction{F}[0](1-\gamma_i/b_i)\right)
      + C\epsilon^{2-\alpha+\eta} \sum_{i=i_0+1}^N \kappa_i w_{N,i}^{\alpha-\eta}(\sigma)\;.
  \end{align*}
  Since for each fixed $i$, $\lim_{N\to\infty} w_{N,i}=0$, the first sum tends to 0 as
  $n\to\infty$. By  \Cref{lem:bound-recursion-toeplitz} applied
    as previously we obtain
  \begin{align*}
    \limsup_{N\to\infty}     \sum_{j=1}^{\pi_N}
    & \esp\left[\|W_{N,k_j}\GNZ_0\|^2 \ind{\|W_{N,k_j}\GNZ_0\|\leq \epsilon}\right]
      = \bigo(\epsilon^{2-\alpha+\eta})   \; .
  \end{align*}
  This concludes the proof of \Cref{eq:ANSJ-convergence-SGD} and of
  \Cref{th:stable_law_convergence}.
\end{proof}


\section{Averaging}
\label{sec:averaging}
We now study averaging of the algorithm, that is normalized sums 
\begin{align*}
  c_N^{-1}  \sum_{i=1}^N(\bstheta_i-\bstheta^*) \; , 
\end{align*}
where the norming sequence $\{c_N\}_{N\geq 1}$ is defined as follows
\begin{align*}
  %
  c_N & = b_N^{-1}\quantilefunction{F}[0]\left(1-\frac{1}{\beta_N b_N^\alpha}\right) \; ,  \ \
        \beta_N  = \sum_{i=1}^Nb_i^{1-\alpha} \; .
\end{align*}
We need to strengthen the assumptions on the learning and batching rates.
\begin{assumption}
  \label{assumption:learning-rate-rv-again}
  The learning rate $\{\gamma_i\}_{i\geq 1}$ is non-increasing and regularly varying with
  index $-\rho$, $\rho\in [0,1)$, the reciprocal of the learning rate sequence
  $\{1/\gamma_i\}_{i\geq 1}$ is concave, the batching rate $\{b_i\}_{i\geq1}$is
  non-decreasing and regularly varying with index $r\geq0$ and
  $\lim_{i\to\infty}\gamma_i/b_i=0$, and
    \begin{align}
      \label{eq:conditions-r-rho}
      r(\alpha-1)<1 \; ,  \ \ \rho\alpha>1-r(\alpha-1) \; .
    \end{align}
\end{assumption}

\begin{theorem}
  \label{th:averaging_stable_convergence}
  Let \Cref{assumption:function_H,assumption:regvar-GNZ,%
    assumption:gradient_noise_new_as,%
    assumption:function_H_taylor_exp,assumption:learning-rate-rv-again} hold with
  $\gamma_1L<1$. Then
  \begin{align*}
    \frac{1}{c_N}\sum_{i=1}^N(\bstheta_i-\bstheta^*)\Rightarrow S_{\infty} \; , 
  \end{align*}
  where $S_\infty$ is a stable random vector with characteristic function given by
  \begin{align*}
    \esp\left[\exp(iu^\top S_\infty)\right]
    &= \exp\left(\int_{\Rset^d\setminus\{0\}}
      \Big(e^{iu^\top x}-1-iu^\top x\Big)\,\overline{\nu}(\rmd x)\right), \\
      \overline{\nu}(\cdot)
      & =\nu(\nabla H(\bstheta^*)\cdot)
  \end{align*}

\end{theorem}
\paragraph{Comments:}
\begin{enumerate}
\item The rate of convergence is independent of the learning rate. In particular, taking
  the constant batch size $b=1$ we recover the rate $N^{1/\alpha}$ (up to a slowly varying
  function) as in \cite{wang:gurbuzbalanan:shu:simsekli:erdogdu:2021}.  However, besides
  batching, we work under much more general assumptions than in the aforementioned
  reference and we also make weaker restrictions on the learning rate.
\item The assumption that the sequence $\{1/\gamma_i\}_{i\geq 1}$ is concave is meant to simplify some technical
  arguments. This is true for the most commonly used learning rates $\gamma_i=ci^{-\rho}$,
  $0\leq \rho < 1$ or $\gamma_i=c/\log(i)$.

\item The second condition in \Cref{eq:conditions-r-rho} ensures that certain remainder
  term is negligible. The first one ensures that $c_N\to\infty$. Indeed, since the
  sequence $\{b_i\}_{i\geq 1}$ is regularly varying with index $r\geq 0$, $\beta_N$ is regularly
  varying with index $1-(\alpha-1)r$ and $ \beta_N b_N^\alpha$ is regularly varying with
  index $r+1$. Since the map $x\mapsto \quantilefunction{F}[0](1-1/x)$ is regularly
  varying at infinity with index $1/\alpha$, we obtain by the composition of regularly
  varying functions, cf.~\citet[Proposition~1.1.9]{kulik:soulier:2020}, that
  $\quantilefunction{F}[0](1-1/b_N^\alpha\beta_N)$ is regularly varying with index
  $(r+1)/\alpha$ and finally that $c_N$ is regularly varying with index
  $(1-r(\alpha-1))/\alpha)$.  More precisely, writing
  ${1-F_0(x)}=\frac1{x^{\alpha}\ell_0(x^{\alpha})}$ with $\ell_0$ slowly varying, we have
  \begin{align}
    \label{eq:scaling_sequence_averaging}
    c_N = \left(\beta_N\ell_0^{\#}(\beta_Nb_N^{\alpha})\right)^{\frac{1}{\alpha}},
  \end{align}
  where $\ell_0^{\#}$ is a de Bruijn conjugate of $\ell_0$. The properties
    $\beta_N\to\infty$ and $c_N\to\infty$ can still hold if $r(\alpha-1)=1$ but depends on
    the slowly varying functions appearing in $F_0$ and the sequence $\{b_i\}_{i\geq 1}$. Considering this case would lead to technicalities of no practical interest.
\item The presence of the term $b_N^\alpha$ inside the slowly varying function in
  \Cref{eq:scaling_sequence_averaging} is significant.  For instance, if $b_i=i^r$, $r>0$
  and $\ell_0(x)=\log^\kappa(x)$, $\kappa\not=0$, a de Bruijn conjugate is then
  $1/\log^\kappa$ and
  \begin{align*}
    \lim_{N\to\infty}\frac{\ell_0^{\#}(\beta_N b_N^\alpha)}
    {\ell_0^{\#}(\beta_N)} = \frac{1+r}{1+r-r\alpha} \; .
  \end{align*}
  With the same $b_i$ and $\ell_0(x)=\exp(c \log^\kappa(x))$ with $\kappa\in (0,1)$ and
  $c\not=0$ (which is slowly varying), we have
  \begin{align*}
    \lim_{N\to\infty}\frac{\ell_0^{\#}(\beta_N b_N^\alpha)}  {\ell_0^{\#}(\beta_N)}
    =
    \begin{cases}
      0 & \mbox{ if } c<0 \; , \\
      +\infty& \mbox{ if }c>0 \; .
    \end{cases}
  \end{align*}
\end{enumerate}
\paragraph{Complexity comparison.}
Consider $b_i=i^r$, $r\geq 0$ and $\gamma_i=i^{-\rho}$, $\rho\in (0,1)$. Then the resulting sample size is $n=\sum_{i=1}^N b_i\approx \frac{1}{r+1}N^{r+1}$ and the scaling factor $c_N$ is of the order 
$$
C_N\approx N^{(1-r(\alpha-1))/\alpha}\approx n^{B(\alpha,r)}\;, \ \ B(\alpha,r)=\frac{1-r(\alpha-1)}{\alpha(1+r)}\;.
$$ 
The function $r\to B(\alpha,r)$ is strictly decreasing in $r$. Thus, for fixed sample size $n$, batching with $r$ arbitrarily close to $1/(\alpha-1)$ gives the shortest confidence intervals for $\bstheta^*$. 
\begin{proof}[Proof of \Cref{th:averaging_stable_convergence}] 
Set $\Delta_i=\bstheta_i-\bstheta^*$
and (as previously) $A=\nabla H(\bstheta^*)$. As in the proof of
\Cref{th:stable_law_convergence} we rewrite the algorithm as
\begin{align*}
  \Delta_{i+1}=(\id-\gamma_{i+1}A)\Delta_i
  +\gamma_{i+1}\gradientnoise_{i+1}(\bstheta^*)+\gamma_{i+1}R_{i+1}\;,
\end{align*}
with
\begin{align}
  \label{eq:Rn-decomp}
  R_{i+1}=R_{i+1}^{(1)}+R_{i+1}^{(2)} \; , 
\end{align}
 $R_{i+1}^{(1)}$ and $R_{i+1}^{(2)}$ defined in \eqref{eq:Rn1} and \eqref{eq:Rn2}.
Consequently,
\begin{align}
  \sum_{i=1}^N\Delta_{i}
  & = \sum_{i=1}^N\sum_{j=1}^{i}\left(\prod_{k=j+1}^{i}(\id-A\gamma_{k})\right)
    \left(\gamma_j\gradientnoise_j(\bstheta^*)
    +\gamma_jR_j\right)+\sum_{i=1}^N\prod_{k=1}^i(\id-A\gamma_{k})\Delta_0
  \nonumber \\
  \label{eq:sum_averaging}
    &= S_N + U_N + V_N \; , 
\end{align}
with
\begin{align*}
  S_N & = \sum_{i=1}^N\sum_{j=i}^{N}\left(\prod_{k=i+1}^{j}(\id-A\gamma_{k})\right)
       \gamma_i\gradientnoise_i(\bstheta^*)\;,\\
  U_N & =\sum_{i=1}^N\prod_{k=1}^i(\id-A\gamma_{k})\Delta_0\;,\\
  V_N & = \sum_{i=1}^N\sum_{j=i}^{N}\left(\prod_{k=i+1}^{j}(\id-A\gamma_{k})\right)\gamma_iR_i\;.
\end{align*}
Set $r_i=\sum_{k=1}^i\gamma_k$. By \Cref{assumption:function_H}, \Cref{lemma:function_H} and the inequality
$1-x\leq e^{-x}$ we have
\begin{align*}
  \vvvert\prod_{k=1}^i(\id-A\gamma_{k})\vvvert^p
  \leq \prod_{k=1}^ie^{-\frac{L}{p}\gamma_k}=e^{-\frac{Lr_i}{p}} \; .
\end{align*}
By \Cref{assumption:learning-rate-rv-again}, the sequence $\{r_i\}_{i\geq 1}$ is regularly varying with index
$1-\rho>0$, thus $\sum_{i=1}^{\infty}e^{-\frac{L}{p}r_i}<\infty$ and we have
\begin{align*}
  \sup_{N\geq 1}\|\sum_{i=1}^N\prod_{k=1}^i(\id-A\gamma_{k})\Delta_0\|^p
  \leq \|\Delta_0\|^p \sum_{i=1}^{\infty}e^{-\frac{L}{p}r_i} = O_P(1) \; .
\end{align*}
Thus $U_N=O_P(1)=o_P(c_N)$. \Cref{lem:next_term} proves that $V_N=\smallo_{\pr}(c_N)$.

Let $\pi_n$ and $k_i$ be as in \Cref{eq:debatching_1a} and \Cref{eq:debatching_1b}. Then,
\begin{align*}
  S_N = \sum_{i=1}^N\sum_{j=i}^N \prod_{k=i+1}^j (\id-\gamma_{k}A)
  \gamma_i \gradientnoise_i(\bstheta^*)
  = \sum_{i=1}^N G_{N,i} \gradientnoise_i(\bstheta^*)
  = \sum_{j=1}^{\pi_N}\frac{G_{N,k_j}}{b_{k_j}}\GNZ_j\;,
\end{align*}
with 
\begin{align}
  \label{eq:g_n_j}
  G_{N,i}=\gamma_i\sum_{j=i}^N\prod_{k=i+1}^j(\id-\gamma_{k}A)\;.
\end{align}
Thus, $S_N$ is a weighted sum of independent, regularly varying random vectors. Hence, by
\Cref{theo:triangular_convergence}, we will have proved
\Cref{th:averaging_stable_convergence} if we prove the convergence of the mean measure,
the null array condition and the asymptotic negligibility of small jumps:
\begin{subequations}
  \begin{align}
    \sum_{j=1}^{\pi_N}
    & \ \pr\left(G_{N,k_j}\GNZ_0\in c_Nb_{k_j}\cdot\right) \convvague \overline{\nu} \; ,
      \label{eq:meanmeasure-averaging} \\
    \lim_{\epsilon\to0} \limsup_{N\to\infty}
     \frac1{c_N^2}
      \sum_{j=1}^{\pi_N}  \frac{1}{b_{k_j}^2} & \ \esp\left[\|G_{N,k_j}\GNZ_0\|^2
      \ind{\|G_{N,k_j}\GNZ_0\|\leq \epsilon c_Nb_{k_j}}\right] = 0 \; ,
      \label{eq:ANSJ-averaging} \\
    \lim_{N\to\infty} \max_{1 \leq i\leq N}
    & \ \pr(\|G_{N,i}\GNZ_0\| > \epsilon b_i c_N) = 0 \; ,
                                               \label{eq:nullarray-averaging}
  \end{align}
\end{subequations}
for all $\epsilon>0$.
\end{proof}

\begin{proof}[Proof of \eqref{eq:meanmeasure-averaging}]
  Define the measures
  \begin{align*}
    \nu_{t}=\frac{\pr\left(\GNZ_0 \in t \,\cdot\right)}
    {\pr\left(\|\GNZ_0\| > t\right)}\;.
  \end{align*}
  By regular variation, $\nu_t\convvague\nu$ (the exponent measure of $\GNZ_0$). Since
  $c_N\to\infty$ and $b_i\geq1$, $\nu_{b_ic_N}$ converges vaguely to $\nu$ uniformly
  \wrt~$i$.  We must prove that for every set $S$ separated from $0$ in $\Rset^d$ such
  that $AS$ is a continuity set of $\nu$,
  \begin{align}
    \label{eq:mean-measure-d>1-one_averaging}
    \lim_{N\to\infty}\sum_{i=1}^Nb_i\,\pr\left(\frac{G_{N,i}}{b_i c_N} \GNZ_0 \in S\right)=
    \lim_{N\to\infty}\sum_{i=1}^Nb_i \overline{F_0}(b_i c_N)\nu_{b_ic_N}(G_{N,i}^{-1}S) =\nu(AS)\;.
  \end{align}
  The key to proving this convergence is to realize that (contrary to the situation
    in the proof of \Cref{th:stable_law_convergence}) only terms of the sum with not too
  large indices contribute to the limit, namely those with indices $i\leq N-T/\gamma_N$
  for some arbitrary $T>0$. We will prove that
  \begin{subequations}
    \begin{align}
      \label{eq:lim_T_lim_sup_N_averaging} 
      \lim_{T\to \infty} \limsup_{N\to\infty}
      & \sum_{i=N-T/\gamma_N}^N  b_i\pr\left(\frac{G_{N,i}}{b_i c_N}\GNZ_0\in S\right)=0 \; , \\
      \label{eq:mainterm_averaging}
      \lim_{N\to\infty} & \sum_{i=1}^{N-T/\gamma_N}b_i\pr\left(\frac{G_{N,i}}{b_i c_N}
                          \GNZ_0\in S\right)=\nu(AS)\;.
    \end{align}
  \end{subequations}
  Since $S$ is separated from 0 and by \Cref{lem:bound-on-matrices} there exists
  $\epsilon>0$ (depending on $S$) such that for $T$ large enough, we have
  \begin{align*}
    \sum_{i=N-T/\gamma_N}^N b_i \pr\left(\frac{G_{N,i}}{b_i c_N}\GNZ_0\in  S\right)
    & \leq \sum_{i=N-T/\gamma_N}^N b_i
      \pr\left(\|\GNZ_0\|\geq \epsilon b_i c_N\right) 
     \leq \frac{T b_N\overline{F_0}(\epsilon b_{N-T/\gamma_N} c_N)}{\gamma_N}
      \; .
  \end{align*}
  In the latter inequality we used the monotonicity of $\{b_i\}_{i\geq 1}$. By assumption,
  $\gamma_N$ is regularly varying with index $\rho\in[0,1)$, so $\gamma_N=o(N)$. Since
  $\{b_N\}_{N\geq 1}$ is regularly varying with index $r$, $\{c_N\}_{N\geq 1}$ is regularly varying with index
  $(1-r(\alpha-1))/\alpha$ this yields that
  $$b_N\overline{F_0}(\epsilon b_{N-T/\gamma_N} c_N)/\gamma_N \sim
  b_N\overline{F_0}(\epsilon b_{N} c_N)/\gamma_N$$ and the latter sequence is regularly
  varying with index $\rho-1$, hence converges to $0$ as $N\to\infty$. This proves
    \Cref{eq:lim_T_lim_sup_N_averaging}.
  Fix $\epsilon>0$. Let $i_0$ be such that \Cref{eq:matrixGNilowerbound} holds for
  large enough $N$ and $i_0\leq i \leq N$. Then there exists an $\epsilon$-neighborhood
  $S^\epsilon$ of $S$ such that $G_{N,i}^{-1}S\subset A S^\epsilon$ for all such indices
  $i$.  This yields 
  \begin{align*}
    \sum_{i=i_0}^{N-T/\gamma_N}b_i \overline{F_0}(b_i c_N)\nu_{b_ic_N}(G_{N,i}^{-1}S)
    &   \leq \sum_{i=i_0}^{N-T/\gamma_N}b_i \overline{F_0}(b_i c_N)\nu_{b_ic_N}(A S^\epsilon)\;.
  \end{align*}
  Since $\nu_{b_ic_N}$ converges vaguely to $\nu$ uniformly w.r.t. $i$, by the Portmanteau
  theorem \cite[Theorem B.1.17]{kulik:soulier:2020}, {we obtain that for any
    $\eta>0$ and $N$ large enough
    \begin{align*}
      \nu_{b_ic_N}(A S^\epsilon)    &  \leq \nu(\overline{AS^\epsilon}) + \eta \; .
    \end{align*}  
    where $\overline{E}$ denotes the closure of a set  $E$. By
    \Cref{lem:carac-c_N} and since $\eta$ is arbitrary, this yields}
    \begin{align*}
      \limsup_{N\to\infty} \sum_{i=i_0}^{N-T/\gamma_N}
      b_i \overline{F_0}(b_i c_N)\nu_{b_ic_N}(G_{N,i}^{-1}S)
      \leq \nu(\overline{AS^\epsilon}) \; .
    \end{align*}
  Furthermore, for any fixed $i_0$, since the matrices $G_{N,i}$, $i=1,\ldots,i_0$ are
  uniformly bounded away from zero,
  \begin{align*}
    \sum_{i=1}^{i_0}b_i \pr(G_{N,i}\GNZ_0\in b_i c_N S)\leq i_0 \pr(\|\GNZ\|_p> c_N\epsilon)\to 0 
  \end{align*}
  as $N\to\infty$. Therefore, combining the last two displays,
  \begin{align*}
    \limsup_{N\to\infty}\sum_{i=1}^{N-T/\gamma_N}b_i\pr\left(\frac{G_{N,i}}{b_i c_N}\GNZ_0\in S\right)
    \leq \nu(\overline{AS^\epsilon})\;.
  \end{align*}
  Since $\epsilon$ is arbitrary, if $AS$ is a continuity set of $\nu$, this yields
  \begin{align*}
    \limsup_{N\to\infty}\sum_{i=1}^{N-T/\gamma_N}b_i\pr\left(\frac{G_{N,i}}{b_i c_N}\GNZ_0\in S\right)
    \leq \nu(AS)\;.
  \end{align*}
  On the other hand, if $S$ has a non empty interior $S^\circ$, then we can define the set
  $S_\epsilon^\circ$ by
  \begin{align*}
    S_\epsilon^\circ=\bigcup_{x,r: B(x,r)\subset S}B(x,r-\epsilon) \; .
  \end{align*}
  Then, $A^{-1}x \in S_\epsilon^\circ$ implies that $G_{N,i}x\in S^\circ$ and therefore,
  by similar arguments as previously and \Cref{lem:bound-on-matrices}, we obtain, for $S$
  such that $AS$ is a continuity set of $\nu$,
  \begin{align*}
    \liminf_{N\to\infty}\sum_{i=1}^{N-T/\gamma_N}b_i\pr\left(\frac{G_{N,i}}{b_i c_N}\GNZ_0\in S\right)
    \geq \nu((AS)^\circ)=\nu(AS)\;.
  \end{align*}
  This proves \eqref{eq:mainterm_averaging} and concludes the proof of
  \Cref{eq:meanmeasure-averaging}.
\end{proof}

\begin{proof}[Proof of \eqref{eq:ANSJ-averaging}]
  By \Cref{lem:bound-on-matrices}, the matrix norms $\vvvert G_{N,i}\vvvert$ are uniformly
  bounded. Applying successively this property, Karamata Theorem for truncated moments
  \cite[Proposition~1.4.6]{kulik:soulier:2020}, Potter's theorem
  \cite[Theorem~1.5.6]{bingham:goldie:teugels:1989} with some arbitrarily small $\delta$
  and \Cref{lem:carac-c_N}, we obtain, for some constants $C$, $C'$, $C''$,
  \begin{align*}
    \frac1{c_N^2} \sum_{j=1}^{\pi_N}  \frac{1}{b_{k_j}^2}
    & \esp\left[\|G_{N,k_j}\GNZ_0\|^2
      \ind{\|G_{N,k_j}\GNZ_0\|\leq \epsilon c_Nb_{k_j}}\right] \\
    & \leq \frac{C}{c_N^2}\sum_{i=1}^{N} b_i^{-1}\esp[\|\GNZ_0\|_p^2
      \ind{\|\GNZ_0\|_p\leq \epsilon b_i c_N/C}]  \\
    & \leq C'\epsilon^2 \sum_{i=1}^{N} b_i   \overline{F}_0( \epsilon b_i c_N/C)  
     \leq C''\epsilon^{2-\alpha-\delta}
      \sum_{i=1}^{N} b_i   \overline{F}_0(b_i c_N)  = \bigo(\epsilon^{2-\alpha-\delta}) \; .
  \end{align*}
  Letting $\epsilon\to0$ concludes the proof.
\end{proof}
\begin{proof}[Proof of \eqref{eq:nullarray-averaging}]
  By \Cref{lem:bound-on-matrices}, $\vvvert G_{N,i} \vvvert$ is uniformly bounded, hence
  there exists constant $C$ such that for any $\epsilon>0$,
  \begin{align*}
    \pr(\|G_{N,i}\GNZ_0\| > \epsilon b_i c_N)\leq 
    \overline{F_0}\left(\frac{\epsilon b_i c_N}{C}\right).
  \end{align*}
  The latter converges to 0 as $N\to\infty$ uniformly in $i$.
\end{proof}



\appendix

\section{Lemmas}
\label{sec:appendix-technical-details-convergence}

\begin{lemma}
  \label{lemma:p_norm_triangle_inequality}
  Let $p\in[1,2]$. For any $\mathbf{x,y}\in \Rset^d$,
  $\|\mathbf{x}+\mathbf{y}\|^p\leq\|\mathbf{x}\|^p+4\|\mathbf{y}\|^p
  +p\mathbf{y}^T\mathbf{x}^{\langle p-1 \rangle}$.
\end{lemma}
 \begin{proof}
   See \citet[Lemma~8]{wang:gurbuzbalanan:shu:simsekli:erdogdu:2021}.
 \end{proof}

\begin{lemma}\label{lemma:function_H}
  \Cref{assumption:function_H} holds if and only if there exists $p\geq1$ and $\delta,L>0$
  (which may depend on $p$) such that for all $t \in [0,\delta]$ and for all $\bstheta\in\Theta$,
  \begin{align}
     \label{eq:lemma_functionH_1}
    \vvvert \id - t\nabla H(\bstheta) \vvvert^p \leq 1 - Lt \; .
  \end{align}
  If \eqref{eq:lemma_functionH_1} holds for one $p$, then it holds for all $p\geq1$.
\end{lemma}
\begin{proof}
Let  \Cref{assumption:function_H}  hold. Let $\bstheta\in\Theta$ and $\sigma$ be an eigenvalue of $\nabla H(\bstheta)$ and $\mathbf{x}$ a unit norm eigenvector. Let $t\leq 1/b$. Then for any $p\geq 1$ we have
$$
\|(\id-t \nabla H(\bstheta))\mathbf{x}\|^p=
(1-t \sigma)^p  \le (1-ta)^p\leq 1-ta\;.
$$
Thus \Cref{eq:lemma_functionH_1} holds for $\delta=1/b$ and $L=a$ with $p=1$,
    hence for all $p\geq1$.

Conversely, let \eqref{eq:lemma_functionH_1} holds with $p\geq 1$. 
  Let $\bstheta\in\Theta$ and $\sigma$ be an
  eigenvalue for $\nabla H(\bstheta)$ and $\mathbf{x}$ a unit norm eigenvector. Then
    \begin{align*}
      (1-t\sigma)^p =  \|(\id-t\nabla H(\bstheta))\mathbf{x}\|^p
      \leq 1-Lt \; .
    \end{align*}
    Since this inequality must hold for all $t\in[0,1/L)$, this implies that
    $\sigma\geq L/p$. On the other hand, we have
    \begin{align*}
      \delta\sigma-1  \leq       |1-\delta\sigma| =  \|(\id-\delta\nabla H(\bstheta))\mathbf{x}\|
      \leq (1-L\delta)^{1/p} \leq 1 \; . 
    \end{align*}
    Thus $\sigma \leq 2/\delta$.  Consequently, \eqref{eq:lemma_functionH_1} implies that
    the eigenvalues of $\nabla H(\bstheta)$ are uniformly bounded above and away from zero
    over $\Theta$. 
\end{proof}

\begin{lemma}
  \label{lem:bound-recursion-toeplitz}
  Let $\{u_i\}_{i\geq1}$, $\{v_i\}_{i\geq1}$ and $\{z_i\}_{i\geq1}$ be sequences of non-negative
  real numbers. Assume that $v_n>0$ for all $n\geq1$, define $w_0>0$,
  $w_n=u_n\prod_{i=1}^{n} v_i^{-1}$, $n\geq1$, and assume that $\{w_n\}_{n\geq 0}$ is ultimately
  non-decreasing and
  \begin{align}
    \label{eq:deltaw}
    \frac{z_nu_n}{u_n-u_{n-1}v_n}=\bigo(1) \; .
  \end{align}
  Then
  \begin{align}
    \label{eq:bigo_u_n}
    \sum_{i=1}^n z_i u_i\prod_{j=i+1}^nv_i = \bigo(u_n)  \; .
  \end{align}
  Let $\{\delta_n\}_{n\geq0}$ be a non-negative sequence such that
  \begin{align}
    \label{eq:recursion-delta}
    \delta_{i+1}\leq v_{i+1}\delta_i+z_{i+1}u_{i+1} \; , \ \ i \geq 0 \;.
  \end{align}
  Then $\delta_n=\bigo(u_n)$.
\end{lemma}

\begin{proof}
  Without loss of generality, we can assume that $\{w_n\}_{n\geq0}$ is increasing.  By the
  monotonicity of $\{w_n\}_{n\geq 0}$ and \Cref{eq:deltaw}, we have
  \begin{align*}
    \sum_{i=1}^{n} z_iu_i \prod_{j=i+1}^n v_j
    & = \frac{u_n}{w_n}\sum_{i=1}^{n} z_i w_i 
      = \frac{u_n}{w_n}\sum_{i=1}^{n} \frac{z_i u_i}{u_i-u_{i-1}v_i}(w_i-w_{i-1})
      = \bigo(u_n) \; .
  \end{align*}
  This proves \Cref{eq:bigo_u_n}.  Next, from the recursion \Cref{eq:recursion-delta} and
  \Cref{eq:bigo_u_n}, we obtain
  \begin{align*}
    \delta_n
    & \leq \delta_0\prod_{i=1}^n v_i + \sum_{i=1}^n z_iu_i\prod_{j=i+1}^n v_i 
      = \frac{\delta_0u_n}{w_n}  + \bigo(u_n)  = \bigo(u_n)
  \end{align*}
  In the last line we used again the monotonicity of $\{w_n\}_{n\geq 0}$.
\end{proof}

\begin{lemma}
  \label{lem:drift}
  The drift $\tilde\gamma$ in \eqref{eq:drift} can be expressed as 
  \begin{align*}
    \widetilde{\gamma}=\int_{\|x\|>1} x\widetilde{\nu}(\rmd x) \; .
  \end{align*}
\end{lemma}

\begin{proof}
  Since $A$ is positive definite, $ \int_{0}^{\infty} e^{-tA}\,\rmd t=A^{-1}$.  Thus, with
  $B=e^{-A}$,
  \begin{align*}
    &\int_{t=0}^{\infty}
     \int_{\mathbb{R}^d} B^t x\left(\ind{\|B^t x\|\le 1}
      -\ind{\|x\|\le 1} \right) \nu(\rmd x) \rmd t \\
    &  = \lim_{\epsilon\to 0}\left\{
      \int_{t=0}^{\infty} \int_{\|x\|>\epsilon} B^t x
      \ind{\|B^t x\|\le 1}\nu(\rmd x)\rmd t     -\int_{t=0}^{\infty}\int_{\epsilon<\|x\|\le 1}
      B^t x\,\nu(\rmd x)\,\rmd t  \right\}  \\
    & = \lim_{\epsilon\to 0}\left\{
      \int_{t=0}^{\infty} \int_{\epsilon<\|x\|} B^t x\,\nu(\rmd x)\,\rmd t
      - \int_{t=0}^{\infty}\int_{\|B^t x\|>1,\ \|x\|>\epsilon} B^t x\,\nu(\rmd x)\,\rmd t     -A^{-1}\int_{\epsilon<\|x\|\le 1} x\,\nu(\rmd x)  \right\} \\
    & = \lim_{\epsilon\to 0}\left\{
      A^{-1}\int_{\epsilon<\|x\|} x\,\nu(\rmd x)
      +\int_{t=0}^{\infty}\int_{\|B^t x\|>1,\ \|x\|>\epsilon} B^t x\,\nu(\rmd x)\,\rmd t  -A^{-1}\int_{\epsilon<\|x\|\le 1} x\,\nu(\rmd x) \right\} \\
    & = A^{-1}\int_{\|x\|>1} x\,\nu(\rmd x) + \int_{t=0}^{\infty}\int_{\|B^t x\|>1}
      B^t x\,\nu(\rmd x)\,\rmd t \\
    & = A^{-1}\int_{\|x\|>1} x\,\nu(\rmd x)
      + \int_{t=0}^{\infty}\int_{\|x\|>1} x\,\nu(B^{-t}\rmd x)\,\rmd t \\
    & = A^{-1}\int_{\|x\|>1} x\,\nu(\rmd x) + \int_{\|x\|>1} x\,\widetilde{\nu}(\rmd x) \; .
  \end{align*}
\end{proof}

\begin{lemma}
  \label{lem:X_n_and_X_n_tilda_convergence}
  Under the assumptions of \Cref{th:stable_law_convergence}, we
  have $$X_N-\widetilde{X}^{\dag}_N=o_P(1), \ \ N\to\infty\;.$$
\end{lemma}
\begin{proof}
  Recall that \Cref{assumption:regvar-GNZ,assumption:gradient_noise_new_as} imply
  \Cref{assumption:gradient_noise} for all $p\in(1,\alpha)$. Therefore, we can apply
  \Cref{th:p_framework_convergence_b>1} for any $p$ arbitrarily close to $\alpha$.  Let
  $\Delta_{N}=X_N-\widetilde{X}_N^\dag$. Then
  \begin{align*}
    \Delta_{N+1}=(\id-\gamma_{N+1}A)\Delta_{N}
    +w_{N+1}\gamma_{N+1}R^{(1)}_{N+1}+w_{N+1}\gamma_{N+1}R_{N+1}^{(2)}+R_{N+1}^{(3)}\;,
  \end{align*}
  where $R_{N+1}^{(1)}$ and $R_{N+1}^{(2)}$ are defined in \Cref{eq:Rn1} and
  \Cref{eq:Rn2},
  $R_{N+1}^{(3)}=\left(w_{N+1}-w_n\right)(\id-\gamma_{N+1}A)(\bstheta_n-\bstheta^*)$. Therefore,
  define the sequences $\{\Delta_{N}^{(i)}\}_{N\geq 1}$, $i=1,2,3$ via the recursions $\Delta_0^{(i)}=0$
  and
  \begin{align*}
    \Delta_{N+1}^{(1)}
    &=(\id-\gamma_{N+1}A)\Delta_{N}^{(1)}+w_{N+1}\gamma_{N+1}R^{(1)}_{N+1},\\
    \Delta_{N+1}^{(2)}
    &=(\id-\gamma_{N+1}A)\Delta_{N}^{(2)}+w_{N+1}\gamma_{N+1}R_{N+1}^{(2)},\\
    \Delta_{N+1}^{(3)}
    &=(\id-\gamma_{N+1}A)\Delta_{N}^{(3)}+R_{N+1}^{(3)}\;.
  \end{align*}
  Note that $\Delta_{N}=\Delta_{N}^{(1)}+\Delta_{N}^{(2)}+\Delta_{N}^{(3)}$. By
  \Cref{assumption:function_H} along with \Cref{lemma:function_H} and the standard norm
  inequality, we have
  \begin{align}
    \label{eq:seq_eq_error_1}
    \esp[\|\Delta_{N+1}^{(1)}\|]
    & \leq(1-\gamma_{N+1}L)^{{1/p}}\esp[\|\Delta_{N}^{(1)}\|]
      +w_{N+1}\gamma_{N+1}\esp[\|R^{(1)}_{N+1}\|]\;,\\
    \label{eq:seq_eq_error_3}
    \esp[\|\Delta_{N+1}^{(3)}\|]
    & \leq(1-\gamma_{N+1}L)^{{1/p}}\esp[\|\Delta_{N}^{(3)}\|]+\esp[\|R^{(3)}_{N+1}\|]
    \;.
  \end{align}
  Let $q\in(1,\alpha)$ and $p\in(q,\alpha)$. By \Cref{assumption:function_H_taylor_exp}
  and \Cref{th:p_framework_convergence_b>1}, we have
  \begin{align}\label{eq:R_N(1)}
    \esp[\|R_{N+1}^{(1)}\|]
    &\leq C\esp[\|\bstheta_N-\bstheta^*\|^q] \leq C\left(\esp[\|\bstheta_N-\bstheta^*\|^p]\right)^{\frac{q}{p}}
      \leq C  \left(\frac{\gamma_{N+1}}{b_{N+1}}\right)^{(p-1)\frac{q}{p}}\;.
  \end{align}
  Again, by \Cref{th:p_framework_convergence_b>1} we have
  \begin{align*}
    \esp[\|R_{N+1}^{(3)}\|]
    \leq w_N\left(\frac{w_{N+1}}{w_N}-1\right)\esp[\|\bstheta_N-\bstheta^*\|]
    \leq (w_{N+1}-w_N)C\left(\frac{\gamma_N}{b_N}\right)^{1-\frac{1}{p}}\;.
  \end{align*}
  Next, applying \Cref{lemma:p_norm_triangle_inequality} and
  \Cref{assumption:gradient_noise}, we have
  \begin{align*}
    \esp[\|\Delta_{N+1}^{(2)}\|^p]
    &\leq(1-\gamma_{N+1}L)\esp[\|\Delta_{N}^{(2)}\|^p]
      +4w^p_{N+1}\gamma^p_{N+1}\esp[\|R^{(2)}_{N+1}\|^p] \\
    &\phantom{\leq + (1}+pw_{N+1}\gamma_{N+1}
      \esp[  (R^{(2)}_{N+1})^\top (\Delta_{N}^{(2)})^{\langle p-1\rangle } ]\;. 
  \end{align*}
  The expectation in the last line vanishes by the same argument as in the proof of
  \Cref{th:p_framework_convergence_b>1}. Indeed, $\Delta_{N}^{(2)}$ is
  $\mathcal{F}_N$-measurable, while the conditional expectation of the $j$th component
  ($j=1,\ldots,d$) of $R^{(2)}_{N+1}$ is zero.  Thus, with $p$ arbitrarily close to
  $\alpha$, we have
  \begin{align}
    \label{eq:seq_eq_error_2}
    \esp[\|\Delta_{N+1}^{(2)}\|^p]
    &\leq(1-\gamma_{N+1}L)\esp[\|\Delta_{N}^{(2)}\|^p]
      +4w^p_{N+1}\gamma^p_{N+1}\esp[\|R^{(2)}_{N+1}\|^p]\;.
  \end{align}

  By \Cref{assumption:gradient_noise_new_as} and \Cref{th:p_framework_convergence_b>1} we
  have
  \begin{align*}
    \esp[\|R_{N+1}^{(2)}\|^p]
    &\leq \frac{C}{b_{N}^{p-1}}\esp[\|\bstheta_N-\bstheta^*\|^p]
      \leq \frac{C}{b_{N}^{p-1}} \left(\frac{\gamma_{N}}{b_n}\right)^{p-1}\;.
  \end{align*}
  
  Therefore, for each inequality in \eqref{eq:seq_eq_error_1}, \eqref{eq:seq_eq_error_3},
  \eqref{eq:seq_eq_error_2}, we apply \Cref{lem:bound-recursion-toeplitz} and obtain
  \begin{align*}
    \esp[\|\Delta_{N+1}^{(1)}\|]
    &=\bigo\left(w_{N+1}\left(\frac{\gamma_{N+1}}{b_{N+1}}\right)^{(p-1)\frac{q}{p}}\right),\\
    \esp[\|\Delta_{N+1}^{(2)}\|^p]
    &=\bigo\left(\frac{\gamma_{N+1}^{p-1}w_{N+1}^p}{b_{N}^{p-1}}
      \left(\frac{\gamma_{N}}{b_N}\right)^{p-1}\right),\\
    \esp[\|\Delta_{N+1}^{(3)}\|]
    &=\bigo\left((w_{N+1}-w_N)
      \left(\frac{\gamma_N}{b_N}\right)^{1-\frac{1}{p}}\gamma_N^{-1}\right).
  \end{align*}
  By \Cref{assumption:regvar-GNZ} the gradient noise is regularly varying with
  index~$\alpha$, thus for any arbitrarily small $\epsilon>0$,
  $w_N=\bigo(b_N/\gamma_N)^{1-1/\alpha+\epsilon}$.  Choose now $p$ and $q$ close enough to
  $\alpha$ so that $(p-1)p/q>\alpha-1-\epsilon$. Then
  \begin{align*}
    w_{N+1}\left(\frac{\gamma_{N+1}}{b_{N+1}}\right)^{(p-1)\frac{q}{p}}
    = \bigo\left(\left(\frac{\gamma_{N+1}}{b_{N+1}}
    \right)^{\frac{(\alpha-1)^2}\alpha-2\epsilon}\right)
    = \smallo(1) \; .
  \end{align*}
  We can also choose $p$ such that $2p-2>2\alpha-2+\epsilon$ and
  $w_{N+1}^p =\bigo(b_{N+1}/\gamma_{N+1})^{\alpha-1+\epsilon }$, thus
  \begin{align*}
    \frac{\gamma_{N+1}^{p-1} w_{N+1}^p}{b_{N}^{p-1}}
    \left(\frac{\gamma_{N}}{b_N}\right)^{p-1}
    = \bigo\left(\left(\frac{\gamma_{N+1}}{b_{N+1}}\right)^{\alpha-1-2\epsilon}\right)
    = \smallo(1) \; .
  \end{align*}
  Since $\{w_i\}_{i\geq 0}$ is regularly varying with index $(1-1/\alpha)(\rho+r)$ and we have assumed
  that $w_{i+1}-w_i=\bigo(w_i/i)$, we obtain
    \begin{align*}
      (w_{i+1}-w_i)\Big(\frac{\gamma_i}{b_i}\Big)^{1-\frac{1}{p}}
    \end{align*}
    is regularly varying with index
    \begin{align*}
      (1/p-1/\alpha)(\rho+r)-1 \; .
    \end{align*}
    Since $\rho<1$, we can choose $p$ large enough so that
    \begin{align*}
      \rho+      (1/p-1/\alpha)(\rho+r)-1 < 0 \; , 
    \end{align*}
    which yields that $\esp[\|\Delta_{N+1}^{(3)}\|]=\smallo(1)$.  Altogether, these
  bounds prove that $\Delta_{N}=\smallo_P(1)$.
\end{proof}

\begin{lemma}
  \label{lem:next_term}
  Under the assumptions of \Cref{th:averaging_stable_convergence},
  \begin{align}
    \label{eq:V_N=o(c_N)}
    V_N=\sum_{i=1}^N\sum_{j=i}^{N}
    \left(\prod_{k=i+1}^{j}(\id-A\gamma_{k})\right)\gamma_i R_i
    = \smallo_{\pr}(c_N)\;.
  \end{align}
\end{lemma}
\begin{proof}
  Recall that $R_{i+1}=R_{i+1}^{(1)}+R_{i+1}^{(2)}$ with $R_{i+1}^{(1)}$ and
  $R_{i+1}^{(2)}$ defined in \Cref{eq:Rn1,eq:Rn2}, 
  and let $G_{N,i}$ be as in \eqref{eq:g_n_j}.  Then
  \begin{align*}
    V_N=    \sum_{i=1}^N  G_{N,i}R_i^{(1)}+\sum_{i=1}^N G_{N,i}R_i^{(2)}\;.
  \end{align*}
  We will study separately the two sums.  By \Cref{lem:bound-on-matrices} the matrices
  $G_{N,i}$ are uniformly bounded from above.  By \eqref{eq:R_N(1)}, we have
  \begin{align*}
    \esp\left[\left\|\sum_{i=1}^N  G_{N,i}R_i^{(1)}\right\|\right]
    & \leq \sum_{i=1}^N \vvvert G_{N,i}\vvvert\esp[\|R_i^{(1)}\|]  
     \leq C \sum_{i=1}^N\left(\frac{\gamma_i}{b_i}\right)^{(p-1)\frac{q}{p}} \; .
  \end{align*}
  The last sum is regularly varying with index $1-(\rho+r)(p-1)q/p$.  Choosing $q$ and $p$
  close enough to $\alpha$, we obtain that
  \begin{align*}
    \sum_{i=1}^N\left(\frac{\gamma_i}{b_i}\right)^{(p-1)\frac{q}{p}}
    = O\left( N^{\max\{1-(\rho+r)(\alpha-1),0\}+\epsilon   }\right)\; ,
  \end{align*}
  for some arbitrarily small $\epsilon$. On the other hand, $\{c_N\}_{N\geq 1}$ is regularly varying
  with index $(1-r(\alpha-1)/\alpha$. The condition \Cref{eq:conditions-r-rho} on the
  rates implies that $1-(\rho+r)(\alpha-1)< (1-r(\alpha-1)/\alpha$, so we conclude that
  \begin{align}
    \sum_{i=1}^N\left(\frac{\gamma_i}{b_i}\right)^{(p-1)\frac{q}{p}}
    & = \smallo(c_N) \; . \label{eq:ass-extra-1}    
  \end{align}
  Next, using \Cref{assumption:gradient_noise_new_as}, the martingale difference structure
  of $R_i^{(2)}$, Doob's inequality and \Cref{th:p_framework_convergence_b>1}, we obtain
  \begin{align*}
    \esp\left[\left\|\sum_{i=1}^N G_{N,i}R_i^{(2)}\right\|^p\right] \leq C \sum_{i=1}^N
    \frac{\esp[\|\bstheta_i-\bstheta^*\|^p]}{b_i^{p-1}} \leq C' \sum_{i=1}^N
    \frac{\gamma_i^{p-1}}{b_i^{2p-2}} \; .
  \end{align*}
  Choosing $p$ close enough to $\alpha$ yields
  \begin{align}
    \sum_{i=1}^N \left(\frac{\gamma_i}{b_i^2}\right)^{p-1}  & = \smallo(c_N^p) \;.
      \label{eq:ass-extra-2}
  \end{align}
  Gathering \Cref{eq:ass-extra-1,eq:ass-extra-2} yields \Cref{eq:V_N=o(c_N)}.
\end{proof}

Let $\preceq$ be the Loewner ordering of symmetric non-negative matrices: for two such
matrices $B,D\in \Rset^{d\times d}$ we write $B\preceq D$ if
$\mathbf{x}^\top B \mathbf{x}\leq \mathbf{x}^\top D \mathbf{x}$ for all
$\mathbf{x}\in\Rset^d$.  Recall that the matrices $G_{N,i}$ are defined in
\Cref{eq:g_n_j}.
\begin{lemma}
  \label{lem:bound-on-matrices}
  Under the assumption of \Cref{th:averaging_stable_convergence}, the matrices $G_{N,i}$
  are uniformly bounded, \ie\
  \begin{align}
    \label{eq:matrixunifbounded}
    \sup_{N\geq1}\sup_{1 \leq i \leq N} \vvvert G_{N,i} \vvvert < \infty \; .
  \end{align}
  Furthermore, for every $\epsilon>0$,
  \begin{itemize}
  \item there exists $T$ such that for $N$ large enough and  $i\leq N-T/\gamma_N$, 
    \begin{align}
      \label{eq:matrixGNilowerbound}x
      (1-\epsilon) A^{-1} \preceq       G_{N,i} \; ; 
    \end{align}
  \item there exists $i_0$ such that for all $i\geq i_0$ and $N\geq i$,
    \begin{align}
      \label{eq:matrixGNiupperbound}
      G_{N,i} \preceq (1+\epsilon) A^{-1} \; .
    \end{align}
  \end{itemize}
\end{lemma}

\begin{proof}
   For $\sigma\in (0,1/\gamma_1)$, write
  \begin{align*}
    g_{N,i}(\sigma)=\gamma_i\sum_{j=i}^N\prod_{k=i+1}^j(1-\sigma\gamma_{k}) \; .
  \end{align*}
  Let $\sigma$ be an eigenvalue of $A$ and $\mathbf{x}$ an associated eigenvector. Then
  \begin{align*}
      G_{N,i}\mathbf{x} = \gamma_i\sum_{j=i}^N\prod_{k=i+1}^j(1-\sigma\gamma_{k}) \mathbf{x} \; .
  \end{align*}
  Thus $\mathbf{x}$ is an eigenvector for $G_{N,i}$ with respect to the
  eigenvalue~$g_{N,i}(\sigma)$. Let $\sigma_0$ be the smallest eigenvalue of $A$. Then
  \begin{align*}
    g_{N,i}(\sigma_0) \leq \gamma_1\sum_{j=1}^N\prod_{k=i+1}^j(1-\sigma_0\gamma_{k})  \; .
  \end{align*}
  This series is summable since $\{\gamma_i\}_{i\geq 1}$ is regularly varying with index
  $\rho\in[0,1)$. This proves that the eigenvalues of $G_{N,i}$ are uniformly bounded,
  therefore \Cref{eq:matrixunifbounded} holds. To prove
  \Cref{eq:matrixGNilowerbound,eq:matrixGNiupperbound}, since $A$ is diagonalizable, we
  only need to prove that for every $\sigma>0$ and $\epsilon>0$, we have
  \begin{align*}
    \frac{1-\epsilon}\sigma \leq \gamma_i\sum_{j=i}^N\prod_{k=i+1}^j(1-\sigma\gamma_{k})
    \leq \frac{1+\epsilon}\sigma \; , 
  \end{align*}
  where the lower bound is valid for $i$ not too large and the upper bound for $i$ not too
  small.

  Since $\{\gamma_i\}_{i\geq 1}$ is regularly varying with index $\rho\in[0,1)$, the series
  $\prod_{k=i+1}^j(1-\sigma\gamma_{k}) $ is summable and we define the sequence
  $\{{g}_i\}_{i\geq 1}$ by
  \begin{align*}
    {g}_i(\sigma)
    =\lim_{N\to\infty}{g}_{N,i}(\sigma) = \gamma_i\sum_{j=i}^\infty
    \prod_{k=i+1}^j(1-\sigma\gamma_{k}) \; .
  \end{align*}
  Then,
  \begin{align*}
    {g}_i(\sigma)
    & = \gamma_i\sum_{j=i}^\infty \frac{1-(1-\sigma\gamma_{j+1})}{\sigma\gamma_{j+1}}
      \prod_{k=i+1}^j(1-\sigma\gamma_{k})\\
    & = \gamma_i\sum_{j=i}^\infty\frac{1}{\sigma\gamma_{j+1}}\prod_{k=i+1}^j(1-\sigma\gamma_{k})
      -   \gamma_i \sum_{j=i}^\infty \frac{1}{\sigma\gamma_{j+1}}
      \prod_{k=i+1}^{j+1}(1-\sigma\gamma_{k}) \\
    & = \gamma_i \sum_{j=i}^\infty \frac{1}{\sigma\gamma_{j+1}}\prod_{k=i+1}^j(1-\sigma\gamma_{k})-
      \gamma_i\sum_{j=i+1}^\infty \frac{1}{\sigma\gamma_{j}}\prod_{k=i+1}^{j}(1-\sigma\gamma_{k}) \\
    & = \frac{\gamma_i}{\sigma} \sum_{j=i}^\infty
      \left(\frac{1}{\gamma_{j+1}} - \frac{1}{\gamma_{j}}\right)
      \prod_{k=i+1}^j(1-\sigma\gamma_{k}) + \frac1{\sigma} \; .
  \end{align*}
  Since $\{1/\gamma_i\}_{i\geq 1}$ is concave, we have 
  \begin{align*}
    g_i(\sigma)
    & \leq  \frac1\sigma\left(\frac{1}{\gamma_{i+1}}-\frac{1}{\gamma_{i}}\right)
      \gamma_i\sum_{j=i}^\infty\prod_{k=i+1}^j(1-\sigma\gamma_{k}) + \frac1{\sigma}  = \frac1{\sigma}\left(\frac{1}{\gamma_{i+1}}-\frac{1}{\gamma_{i}}\right){g}_i(\sigma)
      + \frac1{\sigma} \; .
  \end{align*}
  Also, the concavity of $\{1/\gamma_i\}_{i\geq 1}$ and the fact that $\{\gamma_i\}_{i\geq 1}$ is regularly
  varying with index $\rho\in[0,1)$ imply that
  $\lim_{i\to\infty} (1/\gamma_{i+1}-1/\gamma_i)=0$.  Fix an arbitrarily small
  $\epsilon>0$, set $\epsilon'=\sigma\epsilon/(1+\epsilon)$ and choose large enough $i_0$
  such that $(1/\gamma_{i+1}-1/\gamma_i)\leq\epsilon'$ for all $i\geq i_0$. Then,
    \begin{align}
      g_{N,i}(\sigma) \leq      g_i(\sigma)
      & \leq  \frac{1}{\sigma} \frac1
        {1-\frac1\sigma\left(\frac{1}{\gamma_{i+1}}-\frac{1}{\gamma_{i}}\right)} 
        \leq  \frac{1}{\sigma} \frac1 {1-\frac{\epsilon'}\sigma} = \frac{1+\epsilon}{\sigma} \; .
    \end{align}
 
    For the lower bound, since $\{\gamma_i\}_{i\geq 1}$ is non-increasing, we have, for
    $K\leq N$ and $i\leq K$,
  \begin{align*}
    g_{N,i}(\sigma)
    & \geq \gamma_i\sum_{j=i}^N (1-\sigma\gamma_{j})^{j-i} 
      \geq \gamma_i\sum_{j=i}^N (1-\sigma \gamma_i)^{j-i}\\
    & = \frac{1}{\sigma}\left(1-(1-\sigma\gamma_{i})^{N-i+1}\right) 
     \geq  \frac{1}{\sigma}\left(1-(1-\sigma\gamma_{K})^{N-K+1}\right) \; . 
  \end{align*}
  Choose $K=N-T/\gamma_N$. Then
  \begin{align*}
    \log (1-\sigma\gamma_{K})^{N-K+1} = -\frac{T}{\gamma_N}\log(1-\sigma\gamma_{N-T/\gamma_N})
    \sim \frac{\sigma T\gamma_{N-T/\gamma_N}} {\gamma_N}\; .
  \end{align*}
  By regular variation of $\{\gamma_i\}_{i\geq 1}$, we have
  $\lim_{N\to\infty} \frac{\gamma_N}{\gamma_{N-T/\gamma_N}} =1$. Thus, for $\epsilon>0$
  and large enough $T$, we have, for $i\leq N-T/\gamma_N$, 
  \begin{align*}    
    g_{N,i}(\sigma)  \geq \frac{1-\epsilon}\sigma \; .
  \end{align*}

\end{proof}

\begin{lemma}
  \label{lem:carac-c_N}
  Under \Cref{assumption:learning-rate-rv-again}, 
  \begin{align*}
    \lim_{n\to\infty} \sum_{i=1}^N b_i\overline{F}_0(b_ic_N) = 1 \; .
  \end{align*}
\end{lemma}

\begin{proof}
  By \Cref{eq:scaling_sequence_averaging} and the property \Cref{eq:carac-debruijn} of
  the de Bruijn conjugate, we have
  \begin{align*}
    \sum_{i=1}^N b_i\overline{F}_0(b_ic_N)
    &  = \frac1{\beta_N\ell_0^\#(\beta_Nb_N^\alpha)}  \sum_{i=1}^N \frac{b_i^{1-\alpha}}
      {\ell_0(b_i^\alpha c_N^\alpha)}  
      \sim \frac1{\beta_N}  \sum_{i=1}^N \frac{b_i^{1-\alpha} \ell_0(b_N^\alpha c_N^\alpha)}
      {\ell_0(b_i^\alpha c_N^\alpha)}  \; .
  \end{align*}
  Since $\{b_i\}_{i\geq 1}$ is non-decreasing, for all $i$, by Potter's theorem
  \cite[Theorem~1.5.6]{bingham:goldie:teugels:1989}, we have, for any constant $C>1$,
  $\delta>0$ and $N$ large enough,
  \begin{align*}
    \frac{\ell_0(b_ic_N)}{\ell_0(b_Nc_N)} \leq C \left(\frac{b_i}{b_N}\right)^{-\delta}  \; .
  \end{align*}
  Choose $\delta$ such that $r(\alpha-1+\delta)<1$ and fix $\epsilon>0$. By regular
  variation of the sequence $\{b_i\}_{i\geq 1}$, we obtain
  \begin{align*}
    \limsup_{N\to\infty} \frac1{\beta_N} \sum_{i=1}^{\epsilon N}
    \frac{b_i^{1-\alpha} \ell_0(b_N^\alpha c_N^\alpha)} {\ell_0(b_i^\alpha c_N^\alpha)}
    \leq \limsup_{N\to\infty} \frac{C}{\beta_N} \sum_{i=1}^{\epsilon N} b_i^{1-\alpha-\delta} b_N^\delta
    = \bigo\left( \epsilon^{1-r(\alpha-1+\delta)} \right) \; .
  \end{align*}
  By the Uniform Convergence Theorem,
  \begin{align*}
    \frac{\ell_0(b_ic_N)}{\ell_0(b_Nc_N)}
    =     \frac{\ell_0(\frac{b_i}{b_N}b_Nc_N)}{\ell_0(b_Nc_N)} \to 1 \; ,
  \end{align*}
  uniformly \wrt~$i\geq \epsilon N$. Therefore,
  \begin{align*}
    \limsup_{N\to\infty}   \frac1{\beta_N}
    \sum_{i=\epsilon N}^N b_i^{1-\alpha} \left|\frac{\ell_0(b_N^\alpha c_N^\alpha)}
      {\ell_0(b_i^\alpha c_N^\alpha)}-1 \right| = 0 \; .
  \end{align*}
  Finally,
  \begin{align*}
    \lim_{N\to\infty}   \frac1{\beta_N} \sum_{i=\epsilon N}^N b_i^{1-\alpha}
    = 1 - \epsilon^{1-r(\alpha-1)} \; .
  \end{align*}
  Since $\epsilon$ is arbitrary, this concludes the proof. 
\end{proof}


\end{document}